\documentclass[journal,twoside,web]{ieeecolor}

\usepackage{generic}
\usepackage{cite}
\usepackage{etoolbox}
\makeatletter
\let\ieeeoldmaketitle\maketitle
\def\maketitle{%
  \let\ieeeoldcenterline\centerline
  \def\centerline##1{\par\ieeeoldcenterline{##1}}%
  \ieeeoldmaketitle
}
\makeatother
\usepackage[T1]{fontenc}
\usepackage[utf8]{inputenc}
\usepackage{microtype}
\usepackage{amsmath,amssymb,amsfonts,mathtools}
\usepackage{booktabs,array}
\usepackage{graphicx}
\usepackage{float}
\usepackage{xcolor}
\usepackage{tikz}
\usetikzlibrary{arrows.meta,positioning,shapes.geometric}
\usepackage{url}
\usepackage[hidelinks]{hyperref}
\usepackage[noabbrev,nameinlink]{cleveref}
\usepackage{dblfloatfix}
\usepackage{flafter}

\allowdisplaybreaks

\newtheorem{definition}{Definition}
\newtheorem{lemma}{Lemma}
\newtheorem{theorem}{Theorem}
\newtheorem{corollary}{Corollary}
\newtheorem{proposition}{Proposition}

\newcounter{algorithm}

\newcommand{\R}{\mathbb{R}}
\newcommand{\T}{\mathcal{T}}
\newcommand{\V}{\mathcal{V}}
\newcommand{\E}{\mathcal{E}}
\newcommand{\C}{\mathcal{C}}
\newcommand{\out}{\operatorname{out}}
\newcommand{\pa}{\operatorname{pa}}

\newcommand{\trans}{\mathsf{T}}

\newcommand{\operation}[1]{\par\smallskip\noindent\emph{#1}\enspace}

\def\BibTeX{{\rm B\kern-.05em{\sc i\kern-.025em b}\kern-.08em
    T\kern-.1667em\lower.7ex\hbox{E}\kern-.125emX}}
\begin{document}

\bstctlcite{IEEEtranBSTcontrol}

\title{Rake--Compress Riccati Recursions for\\
Parallel Scenario-Tree Model Predictive Control}
\author{Jo\~ao Sousa-Pinto%
\thanks{Jo\~ao Sousa-Pinto is an independent researcher
(e-mail: joaospinto@gmail.com).}}

\maketitle

\begin{abstract}
Scenario-tree model predictive control (MPC) represents future information by a
rooted tree and optimizes a nonanticipative policy over that tree.  Numerical
methods for solving the resulting nonlinear program typically compute their
search directions through a sequence of branched linear-quadratic regulator
(LQR) subproblems.  The standard tree Riccati recursion requires linear work but
has a dependency chain proportional to tree height.

We present an algebraically exact parallel solver based on rake--compress tree
contraction.  After independent local control condensation, its two operations
act on node and edge data that represent conditional quadratic functions.  A
rake eliminates a leaf and its parent edge, adding their reduced contribution
to the parent-node data.  A compress eliminates a unary node and replaces its
two adjacent edges by one edge, using the same conditional-value composition
as parallel Riccati methods on a chain.
Together they contract an arbitrary rooted tree to its root; reversing the
contraction recovers every Riccati coefficient, state, control, and multiplier.
Given a reusable topology plan, a solve with $N$ nodes and fixed state and
control dimensions has $O(N)$ arithmetic work and storage and $O(\log N)$ span,
independently of tree height, balance, and maximum out-degree.

The formulation allows positive-semidefinite dual regularization, including the
unregularized case, and an exact linear-size lifting covers
the standard scenario-MPC convention of one control per information node.  We
prove the contraction identities and equivalence to the
Karush--Kuhn--Tucker (KKT) system.  Three
MIT-licensed JAX packages implement the bidirectional contraction, the
dual-regularized LQR solver, and a user-facing primal--dual interior-point
solver for tree-structured optimal control.  Dense-KKT
comparisons validate the LQR implementation, accelerator benchmarks show
increasing speedups with tree size, and a constrained autonomous-driving
example exercises the complete stack on an irregular scenario tree.
\end{abstract}


\section{Introduction}
\label{sec:introduction}

Scenario-tree MPC replaces one predicted future by a finite tree of information
histories.  A node represents the information available to the controller.
Decisions are shared while histories coincide and may differ only after a new
observation separates them.  This model is used in stochastic power
dispatch~\cite{patrinos2011dispatch}, dual control~\cite{arcari2020dual}, and
multimodal motion planning~\cite{chen2022branch}.  It also describes
fault-contingent and contact-contingent policies, and closely related tree
structures arise in control of directed physical networks
~\cite{zafar2020tree}.

The resulting trees can be large and highly irregular.  Uncertainty may resolve
at different times on different branches, unlikely modes may be pruned, and
terminal depths may depend on the realized outcome.  After linearization and
quadratic approximation, or after the usual eliminations in a primal--dual
interior-point method, the repeated linear-algebra kernel is a tree-structured
LQR problem~\cite{rao1998ipm,frison2017tree}.  A tree Riccati recursion solves it
with work linear in the number of nodes, but its backward and forward passes
have depth proportional to the longest root-to-leaf path.  Level parallelism
does not shorten a long trunk.

After independent local control condensation, the resulting node and edge data
are closed under the two eliminations required by rake--compress.  A
\emph{rake} eliminates a leaf and its parent edge and adds their reduced
contribution to the parent-node data.  A \emph{compress} eliminates a unary
node and replaces its two adjacent edges by one edge whose data include the
eliminated node term; this uses the associative conditional-value composition
introduced for parallel-scan LQR in~\cite{sarkka2023temporal}.
Rakes expose parallelism across branches; compressions expose it along
branches.  Repeating both operations contracts any rooted tree geometrically,
and reversing them recovers the quantities eliminated during contraction.

Throughout, \emph{exact} means algebraically equivalent in exact arithmetic to
solving the stated KKT system; it does not imply error-free floating-point
computation.

We apply this idea to a dual-regularized branched LQR problem.  The
regularization includes the unregularized dynamics-constrained problem and the
regularized Newton systems that arise in primal--dual interior-point methods.  The algorithm
uses the same topology schedule for factorization and for each subsequent
right-hand-side solve.  For a fixed scenario topology, the schedule is
constructed once and reused.  For a tree with $N$ nodes and fixed local
dimensions, each numerical solve then has
$O(N)$ work and storage and $O(\log N)$ span.  Here \emph{work} denotes the
total number of primitive numerical operations, and \emph{span} the longest
dependency chain under unlimited parallelism.  The numerical algorithm is
work-efficient because its work matches the $O(N)$ sequential recursion.

Every contracted component is represented by a conditional quadratic whose
matrix dimensions depend only on the boundary state dimension.  This
representation is closed
under both leaf and unary-node elimination and retains sufficient
residual information to reconstruct all eliminated primal and dual variables.
The contributions are as follows.

\begin{enumerate}
  \item We derive rake, compress, and reverse-expansion rules for
  dual-regularized Riccati factorization, affine solution, and state recovery.
  Their dimensions depend only on the state and control dimensions,
  not on the size of the represented subtree.  The derivation permits singular
  regularization, including zero.
  \item We prove exact equivalence to the branched KKT system and establish
  $O(N)$ numerical work and storage and $O(\log N)$ span for the resulting
  factorization and solve.
  \item We give an exact linear-size transformation from the usual
  nonanticipative node-control formulation of scenario MPC to the edge-control
  form used by the implementation.
  \item We give a reusable deterministic contraction schedule for arbitrary
  rooted trees.  It removes at least one third of the nodes present at the
  beginning of each structural round, with a tight bound, and its
  readiness-weighted sibling
  reductions yield $O(\log N)$ primitive dependency depth without a bound on
  out-degree.
  \item We provide open-source JAX packages for bidirectional rake--compress
  contraction~\cite{rakecompressjax}, regularized LQR~\cite{regularizedlqrjax},
  and primal--dual interior-point tree control through
  \texttt{primal-dual-lipa}~\cite{primalduallipa}.  We validate the LQR solver
  against independently assembled KKT systems and exercise the complete stack
  on a constrained tree-control example.
\end{enumerate}

\subsection{Related work}

Steinbach~\cite{steinbach2002treesparse} gave an early control-theoretic
formulation of convex programs on general scenario trees, including incoming,
outgoing, and implicit controls, and derived linear-work recursive KKT
algorithms.  An incoming control is local to a node and therefore to its
incoming edge, as in the normal form below.  Blomvall and
Lindberg~\cite{blomvall2002riccati} showed that the direction-finding problem
for the barrier subproblem of a primal interior-point method for multistage
stochastic programming is a linear-quadratic control problem over the scenario
tree.  Later sparse solvers include the tree Riccati factorization
in~\cite{frison2017tree} and the dual-Newton method in
\textsc{treeQP}~\cite{kouzoupis2019treeqp}.  These methods have linear work but
follow the height of the tree.

Temporally parallel Riccati methods for a chain take several distinct forms.
Nielsen and Axehill recursively reduce horizon blocks to a smaller
optimal-control master problem~\cite{nielsen2015parallel}.  Laine and Tomlin
solve endpoint-explicit subproblems in parallel and couple them through a
block-banded system for the link states; their reconstruction assumes this
system has full rank, which they justify under strict convexity
~\cite{laine2019parallel}.  S\"arkk\"a and Garc\'ia-Fern\'andez instead
represent conditional interval values by an associative composition evaluated
with parallel scans~\cite{sarkka2023temporal}.  The last construction is the
chain algebra used by our compress operation; dual-regularized chain
recursions retain this associative structure~\cite{sousapinto2026dual}.

Scenario decomposition exposes independent trajectories by dualizing
nonanticipativity or dynamics~\cite{klintberg2017dual}, and proximal methods can
scale this approach to very large trees~\cite{sampathirao2024massively}.
Hansknecht et al.~\cite{hansknecht2025treecoupled} study the broader class of
tree-coupled saddle-point systems.  Our result concerns the direct
optimal-control factorization: it retains the coupled tree and exploits a
bounded-dimensional Riccati closure to obtain linear work.

Miller and Reif introduced rake--compress tree
contraction~\cite{miller1985tree,miller1989part1}.  They define \textsc{Rake}
to remove all leaves, \textsc{Compress} to halve every maximal unary chain,
and \textsc{Contract} as the simultaneous application of both operations to
the tree.  They prove that $O(\log N)$ applications suffice.  A reduction of
$k$ sibling contributions has $O(\log k)$ local depth with bounded-arity
operations; their asynchronous variant overlaps such reductions with other
ready work and retains $O(\log N)$ total depth, rather than multiplying the
number of applications by $O(\log k)$.  Gazit, Miller, and Teng subsequently
gave a deterministic, work-optimal $O(\log N)$-time implementation for
arbitrary-degree trees on the exclusive-read exclusive-write (EREW) parallel
random-access machine (PRAM),
which forbids concurrent access to one memory
location~\cite{gazit1988optimal}.  Related contractions support marginal-probability and
maximum-a-posteriori queries in probabilistic graphical
models~\cite{delcher1996logarithmic,sumer2011adaptive}.

Our contraction uses the same rake and compress primitives; a new scheduler is
not required by the LQR algebra.  The difference is the rule that selects and
orders those primitives.  Classical \textsc{Contract} simultaneously removes
the current leaves and halves the current maximal unary chains, and its
arbitrary-degree guarantees use either asynchronous readiness or the isolation
and processor-allocation machinery developed for EREW execution.  We could
adapt one of those schedules, but instead use a simpler deterministic static
plan tailored to repeated numerical solves: each structural round first rakes
the leaves and then selects nonadjacent unary centers in the resulting tree;
high-degree rakes are expanded into readiness-weighted binary reductions, and
all elimination dependencies are recorded for reverse traversal.  The plan is
built once from the topology and reused for factorization and multiple
right-hand sides.

This choice changes the scheduler, not the contraction algebra.  In particular,
the resulting dependency graph is neither Miller and Reif's asynchronous
schedule nor Gazit, Miller, and Teng's EREW schedule, so their complexity
theorems cannot be applied to it directly.  \Cref{sec:complexity} therefore
proves geometric progress, linear work, and logarithmic span for the precise
plan used here.  Its dataflow permits concurrent reads but assigns a unique
writer to every output---the concurrent-read exclusive-write (CREW)
discipline.  The schedule is independent of the numerical data; the algorithm
instantiates it with the factorization and solve operations derived below.

The distinction from existing direct Riccati methods is as follows.  A standard
tree Riccati recursion has linear work, and subtrees that are ready at the same
time can be processed concurrently, but the dependency from a leaf to the root
still follows the tree height~\cite{frison2017tree}.  On a chain, recursive
master-problem reduction~\cite{nielsen2015parallel}, endpoint-explicit
decomposition with a full-rank link system~\cite{laine2019parallel}, and
associative conditional-value composition~\cite{sarkka2023temporal} provide
different forms of temporal parallelism.  None of them, by itself, specifies
how to contract a branching tree; our compress operation specifically uses the
last of these algebras.

Zhang et al.~\cite{zhang2025parallelbranch} partition the backward pass at the
last branching time.  The suffix beyond that time is a collection of
nonbranching paths, which they process concurrently using a conditional-value
scan along each path.  For the initial branching portion, their sparse method
applies an ordinary tree Riccati recursion: nodes at one stage can be processed
in parallel, but successive stages remain dependent.  Their alternative
flattens this portion into root-to-leaf paths, condenses the paths concurrently,
reconciles their shared controls, and solves the resulting dense control-space
system.

Our method instead factors the original sparse problem directly on an arbitrary
rooted tree.  It requires neither a distinguished branching time nor any other
prescribed topology pattern: rakes combine sibling subtrees and compressions
compose unary portions wherever they occur.  A topology scheduler coordinates
these eliminations, and the numerical passes operate on bounded-dimensional
Riccati data without flattening the tree into paths or forming a dense
global system.  The resulting exact factorization and solve have linear work
and storage and logarithmic span independently of tree height, balance, and
maximum out-degree.  The constrained nonlinear example in
\cref{sec:case-studies} embeds this solver in the primal--dual interior-point framework
developed in~\cite{sousapinto2026dual}.

\section{Branched dual-regularized LQR}
\label{sec:problem}

\subsection{Edge-control normal form}

Let $\T=(\V,\E,\rho)$ be a finite rooted tree.  Every edge
$e=(i,j)\in\E$ is directed from $i=\pa(j)$ to its child $j$; hence
$|\E|=|\V|-1$.  Write $\C(i)$ for the children of $i$ and
$\out(i)=\{(i,j)\in\E:j\in\C(i)\}$.  For clarity, use uniform dimensions
$x_i,y_i\in\R^n$ and $u_e\in\R^m$; the restriction to uniform dimensions is
not essential.

For symmetric $Q_i$, symmetric $R_e$, and symmetric dual regularizers
$\Delta_i$, define
\begin{align}
\mathcal{L}(x,u,y)
={}&\sum_{i\in\V}\left(\tfrac12 x_i^\trans Q_i x_i+q_i^\trans x_i\right)
\notag\\
&+\sum_{e=(i,j)\in\E}
\left(x_i^\trans M_eu_e+\tfrac12u_e^\trans R_eu_e\right)
\notag\\
&+\sum_{e=(i,j)\in\E}r_e^\trans u_e
\notag\\
&+y_\rho^\trans(c_\rho-x_\rho)-\tfrac12y_\rho^\trans\Delta_\rho y_\rho
\notag\\
&+\sum_{e=(i,j)\in\E}
y_j^\trans(A_ex_i+B_eu_e+c_j-x_j)
\notag\\
&-\tfrac12\sum_{e=(i,j)\in\E}y_j^\trans\Delta_jy_j.
\label{eq:lagrangian}
\end{align}
The branched dual-regularized LQR problem is
\begin{equation}
\max_y\min_{x,u}\;\mathcal{L}(x,u,y).
\label{eq:saddle}
\end{equation}
When $\Delta_i=0$, the maximization enforces the root and edge dynamics exactly.
Positive semidefinite $\Delta_i$ are the dual regularization blocks produced by
regularized primal--dual interior-point methods~\cite{sousapinto2026dual}.  Node probabilities in
a stochastic objective are absorbed into $(Q_i,q_i,M_e,R_e,r_e)$.

The following sufficient assumptions ensure that every local factorization
used below is well posed.

\begin{definition}[Standing convexity assumptions]
\label{def:assumptions}
For every edge $e$, $R_e\succ0$; for every node $i$, $\Delta_i\succeq0$; and
\begin{equation}
Q_i-\sum_{e\in\out(i)}M_eR_e^{-1}M_e^\trans\succeq0.
\label{eq:local-convexity}
\end{equation}
Empty sums are zero.
\end{definition}

\begin{proposition}[Well-posedness]
\label{prop:unique}
Under \cref{def:assumptions}, the KKT matrix for the saddle problem in
\cref{eq:saddle} is nonsingular, and the saddle point is unique.
\end{proposition}

\begin{proof}
Let $H$ be the primal Hessian and $D$ the block diagonal matrix with blocks
$\Delta_i$.  Completing each control square rewrites the quadratic form of $H$
as
\begin{align*}
z^\trans Hz={}&\sum_{i\in\V}x_i^\trans
\left(Q_i-\sum_{e\in\out(i)}M_eR_e^{-1}M_e^\trans\right)x_i\\
&+\sum_{e=(i,j)}\left\|u_e+R_e^{-1}M_e^\trans x_i\right\|_{R_e}^2,
\end{align*}
where $\|v\|_R^2=v^\trans Rv$,
so $H\succeq0$.  Let $C$ denote the dynamics Jacobian.  Its columns associated
with node states contain a permuted block-triangular submatrix with diagonal
$-I$, hence $C$ has full row rank.  If $z=(x,u)\in\ker C$ and $z^\trans Hz=0$,
the homogeneous root constraint gives $x_\rho=0$.  The completed square on each
outgoing edge then gives $u_e=0$, and the homogeneous dynamics give $x_j=0$;
induction down the tree yields $z=0$.  Thus $H$ is positive definite on
$\ker C$.

Now suppose
$\bigl[\begin{smallmatrix}H&C^\trans\\C&-D\end{smallmatrix}\bigr]
\bigl[\begin{smallmatrix}z\\y\end{smallmatrix}\bigr]=0$.
Taking inner products of the two block equations with $z$ and $y$ gives
$z^\trans Hz+y^\trans Dy=0$.  Both terms are nonnegative, so $Hz=0$ and
$Dy=0$; the second block equation gives $Cz=0$.  Positive definiteness on
$\ker C$ implies $z=0$, after which $C^\trans y=0$ and full row rank imply
$y=0$.  The KKT matrix is therefore nonsingular.  Its unique stationary point
is the unique saddle point because the Lagrangian is convex--concave.
\end{proof}

The following KKT equations also serve as the end-to-end correctness
specification.  The node-stationarity equation is imposed for every
$i\in\V$, and the control-stationarity and dynamics equations for every
$e=(i,j)\in\E$:
\begin{subequations}
\label{eq:kkt}
\begin{align}
Q_ix_i+q_i-y_i
 +\sum_{e=(i,j)\in\out(i)}(M_eu_e+A_e^\trans y_j)&=0,
\label{eq:kkt-x}\\
M_e^\trans x_i+R_eu_e+r_e+B_e^\trans y_j&=0,
\label{eq:kkt-u}\\
c_\rho-x_\rho-\Delta_\rho y_\rho&=0,\label{eq:kkt-root}\\
A_ex_i+B_eu_e+c_j-x_j&=\Delta_jy_j.
\label{eq:kkt-dyn}
\end{align}
\end{subequations}

\subsection{Nonanticipative node controls}
\label{sec:lifting}

The edge-control form \eqref{eq:lagrangian} is convenient for a uniform local
algebra, but scenario MPC usually assigns one control $v_i$ to each nonterminal
information node:
\begin{equation}
x_j=A_{ij}x_i+B_{ij}v_i+c_j,\qquad j\in\C(i).
\label{eq:node-control}
\end{equation}
The same $v_i$ appears in every outgoing transition because the next
observation has not yet been revealed.  Thus copying $v_i$ independently onto
those edges would violate nonanticipativity.

Let $\mathcal L\subset\V$ be the leaf set.  The corresponding saddle function is
\begin{align}
\mathcal L_{\rm node}={}&
\sum_{i\in\V}\left(\tfrac12x_i^\trans Q_ix_i+q_i^\trans x_i\right)\notag\\
&+\sum_{i\notin\mathcal L}\left(
x_i^\trans M_iv_i+\tfrac12v_i^\trans R_iv_i+r_i^\trans v_i\right)\notag\\
&+y_\rho^\trans(c_\rho-x_\rho)-\tfrac12y_\rho^\trans\Delta_\rho y_\rho\notag\\
&+\sum_{(i,j)\in\E}y_j^\trans
(A_{ij}x_i+B_{ij}v_i+c_j-x_j)\notag\\
&-\tfrac12\sum_{j\ne\rho}y_j^\trans\Delta_jy_j,
\label{eq:node-lagrangian}
\end{align}
which charges the control cost once at $i$.  A sufficient analogue of the
assumptions in \cref{def:assumptions} is
\begin{equation}
\begin{aligned}
R_i&\succ0 &&(i\notin\mathcal L),\\
\Delta_i&\succeq0 &&(i\in\V),\\
Q_i-M_iR_i^{-1}M_i^\trans&\succeq0 &&(i\notin\mathcal L),\\
Q_i&\succeq0 &&(i\in\mathcal L).
\end{aligned}
\label{eq:native-assumptions}
\end{equation}

The node-control problem has an exact linear-size edge-control representation.
Replace $i$ by a predecision node $i^-$ and, if $i$ is nonterminal, a decision
node $i^+$.  With augmented state dimension $n+m$, define
\begin{equation}
s_{i^-}=\begin{bmatrix}x_i\\0\end{bmatrix},\qquad
s_{i^+}=\begin{bmatrix}x_i\\v_i\end{bmatrix}.
\end{equation}
The decision edge $i^-\to i^+$ carries $v_i$:
\begin{equation}
s_{i^+}=
\begin{bmatrix}I&0\\0&0\end{bmatrix}s_{i^-}
+\begin{bmatrix}0\\I\end{bmatrix}v_i.
\label{eq:decision-edge}
\end{equation}
Each uncertainty edge $(i,j)$ becomes an uncontrolled edge $i^+\to j^-$:
\begin{equation}
s_{j^-}=
\begin{bmatrix}A_{ij}&B_{ij}\\0&0\end{bmatrix}s_{i^+}
+\begin{bmatrix}c_j\\0\end{bmatrix}.
\label{eq:uncertainty-edge}
\end{equation}
A dummy control with zero input and a positive-definite quadratic cost makes
\eqref{eq:uncertainty-edge} fit the uniform edge API; its unique value is zero.
State costs are placed at $i^-$, the original control cost on
$i^-\to i^+$, and the regularization
$\operatorname{diag}(\Delta_j,0)$ on the constraint entering $j^-$.  All other
new costs and regularizers are zero.

\begin{proposition}[Scenario-MPC normal form]
\label{prop:lifting}
The KKT solutions of $\max_y\min_{x,v}\mathcal L_{\rm node}$ and of the lifted
saddle problem in \cref{eq:saddle} are in one-to-one correspondence.  In
particular, projection returns the same $(x_i,v_i,y_i)$, with all dummy controls
equal to zero.  The lifted tree has
$|\V|+|\V\setminus\mathcal L|\leq2|\V|-1$ nodes, does not increase the maximum
branching factor, and makes every nonanticipative
control a single edge variable before branching.  Moreover, the conditions in
\cref{eq:native-assumptions} imply the assumptions in
\cref{def:assumptions} for the lifted data.
\end{proposition}

\begin{proof}
Equation~\eqref{eq:decision-edge} forces $s_{i^+}=(x_i,v_i)$, after which
\eqref{eq:uncertainty-edge} reproduces every child transition with the same
$v_i$.  Conversely, each native primal trajectory defines a unique lifted
trajectory.  The dummy-control stationarity equations set those controls to
zero, and the objectives and regularized constraints then agree term by term.

Lift the multiplier entering $i^-$ as $(y_i,0)$.  Stationarity with respect to
$s_{i^+}$ sets the multiplier on the decision edge to
\[
\begin{bmatrix}
\sum_{j\in\C(i)}A_{ij}^\trans y_j\\[1mm]
\sum_{j\in\C(i)}B_{ij}^\trans y_j
\end{bmatrix}.
\]
Its lower block makes decision-edge control stationarity
\[
M_i^\trans x_i+R_iv_i+r_i+
\sum_{j\in\C(i)}B_{ij}^\trans y_j=0,
\]
which is native control stationarity.  Its upper block, together with
stationarity at $i^-$, gives native state stationarity.  These relations
uniquely determine the lifted multipliers from the native ones and are
reversible without requiring $\Delta_i$ to be nonsingular.

Only nonterminal nodes acquire a decision copy, so the node count is
$|\V|+|\V\setminus\mathcal L|\leq2|\V|-1$.  Each $i^-$ has one child, whereas
$i^+$ has the original children of $i$, proving the branching claim.  Finally,
the left side of \cref{eq:local-convexity} at $i^-$ is
\[
\operatorname{diag}
\left(Q_i-M_iR_i^{-1}M_i^\trans,0\right)\succeq0.
\]
It is zero at every decision node, and at a lifted terminal node it is
$\operatorname{diag}(Q_i,0)\succeq0$.  The lifted regularizers are positive
semidefinite, and every original or dummy edge has a positive-definite control
cost.  Thus \cref{eq:native-assumptions} implies
\cref{def:assumptions}.
\end{proof}

The lifting establishes that the edge-control solver covers the standard
scenario-MPC model without duplicating decisions and preserves linear problem
size, so the complexity results below apply unchanged.  A native node-control
derivation could avoid the augmented state, but would have to carry a decision
shared by several outgoing edges through the sibling reductions and coordinate
its eventual elimination.  The edge-control normal form instead localizes each
control to one transition, yielding a uniform algebra for both rake and
compress operations.

\section{Rake--compress tree contraction}
\label{sec:rake-compress}
\begin{figure*}[!t]
\centering
\begin{tikzpicture}[
  x=1cm,y=1cm,
  v/.style={circle,draw,minimum size=5.2mm,inner sep=0pt,font=\small},
  kept/.style={v,fill=blue!7},
  gone/.style={v,fill=orange!14},
  edge/.style={line width=0.8pt},
  op/.style={-{Latex[length=2mm]},line width=0.9pt},
  note/.style={font=\footnotesize,align=center}
]
\node[note,font=\small\bfseries] at (3.15,1.05)
  {(a) Rake: remove a leaf};
\node[kept] (ri) at (0.7,0.25) {$i$};
\node[kept] (rh) at (0.05,-0.85) {};
\node[gone] (rj) at (1.35,-0.85) {$j$};
\draw[edge] (ri)--(rh);
\draw[edge] (ri)--(rj);
\node[note] at (0.7,-1.45) {before};
\draw[op] (1.85,-0.35)--(2.75,-0.35);
\node[kept] (ri2) at (3.4,0.25) {$i$};
\node[kept] (rh2) at (2.75,-0.85) {};
\draw[edge] (ri2)--(rh2);
\node[note] at (3.1,-1.45) {after};
\node[note,text width=3.0cm] at (5.15,-0.35)
  {remove $j$ and $(i,j)$;\\update the data at $i$};

\node[note,font=\small\bfseries] at (11.15,1.05)
  {(b) Compress: remove a unary node};
\node[kept] (ci) at (8.1,0.35) {$i$};
\node[gone] (cj) at (9.25,-0.35) {$j$};
\node[kept] (ck) at (10.4,-1.05) {$k$};
\draw[edge] (ci)--(cj)--(ck);
\node[note] at (9.25,-1.55) {before};
\draw[op] (10.85,-0.35)--(11.75,-0.35);
\node[kept] (ci2) at (12.35,0.25) {$i$};
\node[kept] (ck2) at (13.75,-0.95) {$k$};
\draw[edge] (ci2)--(ck2);
\node[note] at (13.05,-1.55) {after};
\node[note,text width=2.2cm] at (15.25,-0.30)
  {remove $j$, $(i,j)$, and $(j,k)$;\\create $(i,k)$ and its data};
\end{tikzpicture}
\caption{The two rake--compress primitives.  Orange nodes are removed and blue
nodes remain in the tree.  Reverse expansion later reconstructs each removed
node from the data saved during contraction.}
\label{fig:rake-compress}
\end{figure*}
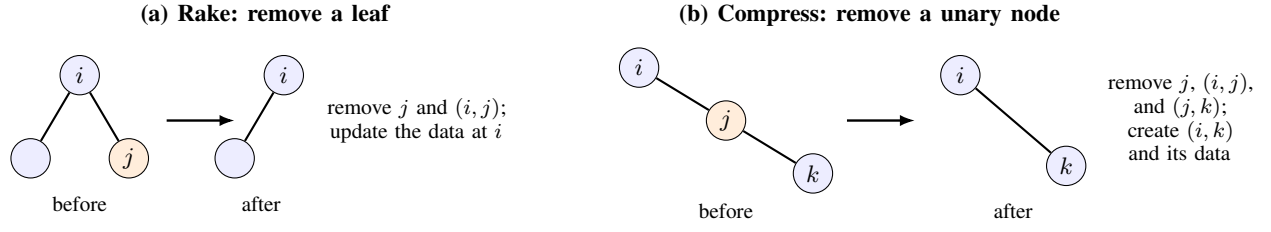

\begin{figure*}[!t]
\centering
\begin{tikzpicture}[
  x=1cm,y=1cm,
  sched/.style={circle,draw,minimum size=6.4mm,inner sep=0pt,
                font=\scriptsize},
  sroot/.style={sched,fill=blue!8},
  srake/.style={sched,fill=orange!18},
  scomp/.style={sched,fill=violet!14},
  sedge/.style={line width=0.75pt},
  stitle/.style={font=\small\bfseries},
  snote/.style={font=\scriptsize,align=center}
]
\node[stitle] at (3.55,6.9) {(a) Chain};
\node[sroot] (a0) at (0.55,5.75) {$\rho$};
\node[scomp] (a1) at (1.55,5.75) {$\mathrm C_1$};
\node[scomp] (a2) at (2.55,5.75) {$\mathrm C_2$};
\node[scomp] (a3) at (3.55,5.75) {$\mathrm C_1$};
\node[srake] (a4) at (4.55,5.75) {$\mathrm R_3$};
\node[srake] (a5) at (5.55,5.75) {$\mathrm R_2$};
\node[srake] (a6) at (6.55,5.75) {$\mathrm R_1$};
\draw[sedge] (a0)--(a1)--(a2)--(a3)--(a4)--(a5)--(a6);
\node[snote] at (3.55,4.95) {three structural rounds};

\node[stitle] at (11.85,6.9) {(b) Star};
\node[sroot] (b0) at (11.85,6.05) {$\rho$};
\node[srake] (b1) at (9.20,4.95) {$\mathrm R_1$};
\node[srake] (b2) at (10.20,4.55) {$\mathrm R_1$};
\node[srake] (b3) at (11.30,4.40) {$\mathrm R_1$};
\node[srake] (b4) at (12.40,4.40) {$\mathrm R_1$};
\node[srake] (b5) at (13.50,4.55) {$\mathrm R_1$};
\node[srake] (b6) at (14.50,4.95) {$\mathrm R_1$};
\foreach \x in {1,...,6} {\draw[sedge] (b0)--(b\x);}
\node[snote] at (11.85,3.92) {one structural round};

\node[stitle] at (3.55,3.35) {(c) Balanced binary tree};
\node[sroot] (c0) at (3.55,2.65) {$\rho$};
\node[srake] (c1) at (2.05,1.75) {$\mathrm R_2$};
\node[srake] (c2) at (5.05,1.75) {$\mathrm R_2$};
\node[srake] (c3) at (1.25,0.65) {$\mathrm R_1$};
\node[srake] (c4) at (2.85,0.65) {$\mathrm R_1$};
\node[srake] (c5) at (4.25,0.65) {$\mathrm R_1$};
\node[srake] (c6) at (5.85,0.65) {$\mathrm R_1$};
\draw[sedge] (c0)--(c1) (c0)--(c2);
\draw[sedge] (c1)--(c3) (c1)--(c4);
\draw[sedge] (c2)--(c5) (c2)--(c6);
\node[snote] at (3.55,0.02) {two structural rounds};

\node[stitle] at (11.85,3.35) {(d) Irregular comb};
\node[sroot] (d0) at (8.55,2.65) {$\rho$};
\node[scomp] (d1) at (9.65,2.30) {$\mathrm C_1$};
\node[srake] (d2) at (10.80,1.90) {$\mathrm R_3$};
\node[srake] (d3) at (9.70,1.15) {$\mathrm R_1$};
\node[scomp] (d4) at (12.00,1.50) {$\mathrm C_1$};
\node[srake] (d5) at (13.20,1.10) {$\mathrm R_2$};
\node[srake] (d6) at (12.10,0.35) {$\mathrm R_1$};
\node[srake] (d7) at (14.40,0.70) {$\mathrm R_1$};
\draw[sedge] (d0)--(d1);
\draw[sedge] (d1)--(d2) (d1)--(d3);
\draw[sedge] (d2)--(d4);
\draw[sedge] (d4)--(d5) (d4)--(d6);
\draw[sedge] (d5)--(d7);
\node[snote] at (14.35,2.22) {three structural rounds};
\end{tikzpicture}
\caption{Representative forward contraction schedules.  Orange nodes are
raked, violet nodes are compressed, and the subscript is the structural round.
Within a round, rakes precede compressions.  The root remains after the
displayed removals, and reverse expansion follows the labels in decreasing
dependency order.}
\label{fig:example-schedules}
\end{figure*}
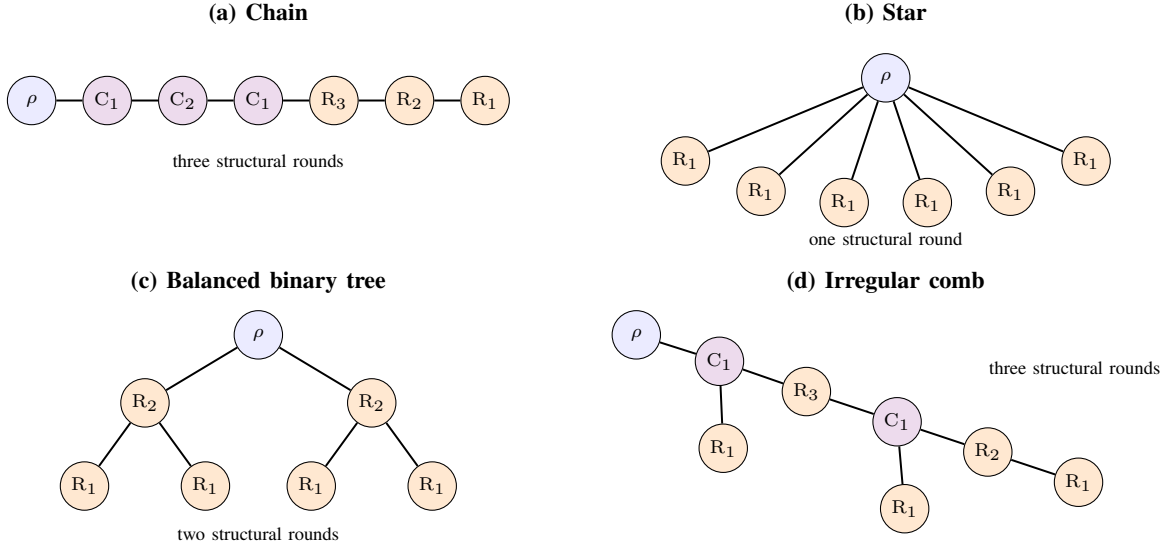

Rake--compress is a topology-only procedure for reducing a rooted tree to its
root and subsequently reconstructing it~\cite{miller1985tree,miller1989part1}.
It is independent of the numerical problem carried by the tree.  Nodes and
edges may carry application-specific data, but the contraction schedule
depends only on the parent--child relation.

The forward contraction uses two primitives, illustrated in
\cref{fig:rake-compress}.
\operation{Rake.}
Remove a nonroot leaf and its parent edge.  A numerical instantiation combines
the data associated with the removed node and edge into the parent-node data.

\operation{Compress.}
Remove a nonroot unary node and replace its incoming and outgoing edges by one
edge joining its predecessor and successor.  A numerical instantiation
combines the data associated with the two removed edges and the removed node
into the new edge data.
Operations on independent leaves or unary centers can be performed
concurrently.  Consecutive compression centers cannot both be selected, but
two compressions may share a retained endpoint: for example, the path
$a\to b\to c\to d\to e$ can become $a\to c\to e$ in one round by
compressing $b$ and $d$.

A structural round first rakes every nonroot leaf and then compresses a
maximal set of pairwise nonadjacent unary centers.  Sibling contributions are
combined without concurrent writes to their parent.  Repeating these rounds
eventually leaves only the root.  Reverse expansion visits the recorded
operations in the opposite dependency order, undoing compressions before
rakes within each round.  \Cref{fig:example-schedules} shows the resulting
patterns on four representative trees.  The labels $\mathrm R_t$ and
$\mathrm C_t$ identify nodes raked and compressed in structural round $t$;
nodes with the same label can be processed concurrently.

This section describes only the topology-level contraction.  The next two
sections define its numerical data for factorization and solve.
Section~\ref{sec:complexity} proves geometric progress of the structural rounds
and constructs the numerical dependency graph, including the reductions of
sibling contributions.  That graph permits operations from different
structural rounds to overlap and is the graph on which span is measured.

\section{Rake--compress factorization}
\label{sec:quadratic-algebra}

We now instantiate the topology operations from
\cref{sec:rake-compress} with conditional quadratic data.
The factorization first eliminates every control locally, then contracts the
resulting node and edge data, and finally reverses the contraction to recover
the Hessians of the reduced subtree value functions required by the solve.
This section derives these operations and proves their correctness and
well-posedness.

The KKT matrix in \cref{eq:kkt} depends only on
$(Q,M,R,A,B,\Delta)$; the vectors $(q,r,c)$ form its right-hand side.
Accordingly, factorization first sets $(q,r,c)=0$ and computes the homogeneous
quadratic factors.  This is not an assumption that the original LQR problem
has zero linear terms.  Section~\ref{sec:affine} subsequently applies the
factorization to arbitrary $(q,r,c)$.  The same separation occurs when the
problem is a Newton subproblem: second derivatives, constraint Jacobians, and
regularization determine the KKT matrix, whereas the current KKT residual
determines the right-hand side.  A single factorization can therefore
serve several Newton right-hand sides, including predictor--corrector steps,
second-order corrections, and iterative-refinement solves.

\subsection{Conditional-value elimination algebra}

At the start of contraction, node $i$ stores
\begin{equation}
U_i=Q_i.
\label{eq:initial-node-data}
\end{equation}
For each original edge $e=(i,j)$, eliminate $u_e$ from the homogeneous edge
term: solve its stationarity equation with respect to $u_e$ and substitute the
result.  This gives the edge triple
\begin{equation}
\alpha_e=(\bar A_e,\bar C_e,\bar P_e),
\label{eq:edge-triple}
\end{equation}
where
\begin{align}
\bar A_e&=A_e-B_eR_e^{-1}M_e^\trans,\notag\\
\bar C_e&=\Delta_j+B_eR_e^{-1}B_e^\trans,\notag\\
\bar P_e&=-M_eR_e^{-1}M_e^\trans.
\label{eq:edge-components}
\end{align}
These edgewise eliminations are independent and are the only control
eliminations performed during contraction.  Moreover,
$\bar C_e=\bar C_e^\trans\succeq0$ under \cref{def:assumptions}.

At any stage of contraction, each node $i$ carries a matrix $U_i$, and each
edge $e$ carries a triple
$\alpha_e=(\mathcal A_e,\mathcal C_e,\mathcal P_e)$, initialized by
$(\mathcal A_e,\mathcal C_e,\mathcal P_e)
=(\bar A_e,\bar C_e,\bar P_e)$.  A rake of leaf $j$ with
parent $i$ eliminates node $j$, its state variable $x_j$, and edge $(i,j)$;
the reduced value of the eliminated node and edge terms is added to $U_i$.  A
compression of a unary node $j$ eliminates node $j$ and $x_j$, and replaces
its incoming and outgoing edges by one edge between the retained endpoints.

For an edge $e=(i,j)$, the triple $\alpha_e$ represents the extended-real-valued
conditional quadratic function
\begin{equation}
\Phi_e(x_i,x_j)=\tfrac12x_i^\trans \mathcal P_ex_i+
\sup_\lambda\left\{\lambda^\trans(\mathcal A_ex_i-x_j)
-\tfrac12\lambda^\trans \mathcal C_e\lambda\right\}.
\label{eq:path-value}
\end{equation}
Stationarity in $\lambda$ gives
$x_j=\mathcal A_ex_i-\mathcal C_e\lambda$.  If $\mathcal C_e$ is singular, the supremum is $+\infty$ unless
$\mathcal A_ex_i-x_j\in\operatorname{range}\mathcal C_e$; hence the representation
includes exact dynamics and does not require $\mathcal C_e^{-1}$.

For a rake of leaf $j$ with parent edge $e=(i,j)$, the algebraic update combines
$\alpha_e$ with the node data $U_j$.  The resulting
contribution to the parent-node data is
\begin{equation}
\tau(\alpha_e,U_j)=\mathcal P_e+\mathcal A_e^\trans
U_j(I+\mathcal C_eU_j)^{-1}\mathcal A_e.
\label{eq:terminal-fold}
\end{equation}
For a compression at $j$, the data on the consecutive edges
$e_\ell=(i,j)$ and $e_r=(j,k)$ must be combined with the node data $U_j$.
Define
\begin{equation}
\widetilde{\mathcal P}_{e_r}=\mathcal P_{e_r}+U_j,
\qquad
S=I+\mathcal C_{e_\ell}\widetilde{\mathcal P}_{e_r},
\label{eq:compress-intermediates}
\end{equation}
and
\begin{subequations}
\label{eq:composition}
\begin{align}
\mathcal A_{e_{\ell r}}&=\mathcal A_{e_r}S^{-1}\mathcal A_{e_\ell},\label{eq:comp-A}\\
\mathcal C_{e_{\ell r}}&=\mathcal C_{e_r}+\mathcal A_{e_r}S^{-1}
\mathcal C_{e_\ell}\mathcal A_{e_r}^\trans,\label{eq:comp-C}\\
\mathcal P_{e_{\ell r}}&=\mathcal P_{e_\ell}+\mathcal A_{e_\ell}^\trans
\widetilde{\mathcal P}_{e_r}S^{-1}\mathcal A_{e_\ell}.
\label{eq:comp-P}
\end{align}
\end{subequations}
We write
$\alpha_{e_{\ell r}}=\alpha_{e_r}\circ_{U_j}\alpha_{e_\ell}$.

\begin{lemma}[Rake and compress elimination identities]
\label{lem:elimination-identities}
Suppose the edge data above satisfy
$\mathcal C_e,\mathcal C_{e_\ell},\mathcal C_{e_r}\succeq0$ and
$\mathcal P_e=\mathcal P_e^\trans$,
$\mathcal P_{e_\ell}=\mathcal P_{e_\ell}^\trans$.  If $U_j\succeq0$ for the
rake and $\widetilde{\mathcal P}_{e_r}\succeq0$ for the compress, then the
updates are well defined and
\begin{align}
\inf_{x_j}\left\{\Phi_e(x_i,x_j)+\tfrac12x_j^\trans U_jx_j\right\}
&=\tfrac12x_i^\trans\tau(\alpha_e,U_j)x_i,\label{eq:fold-identity}\\
\inf_{x_j}\bigl\{\Phi_{e_\ell}(x_i,x_j)+\tfrac12x_j^\trans U_jx_j
&\notag\\[-1mm]
\qquad{}+\Phi_{e_r}(x_j,x_k)\bigr\}
&=\Phi_{e_{\ell r}}(x_i,x_k).
\label{eq:compose-identity}
\end{align}
Moreover,
\begin{align}
\tau(\alpha_e,U_j)-\mathcal P_e&\succeq0,
&\mathcal C_{e_{\ell r}}&\succeq0,\notag\\
\mathcal P_{e_{\ell r}}-\mathcal P_{e_\ell}&\succeq0.
\label{eq:elimination-psd}
\end{align}
In particular, $\tau(\alpha_e,U_j)$,
$\mathcal C_{e_{\ell r}}$, and $\mathcal P_{e_{\ell r}}$ are symmetric.
The identities hold as equalities of extended-real-valued convex functions
(whose values may be $+\infty$) when $\mathcal C_e$,
$\mathcal C_{e_\ell}$, or $\mathcal C_{e_r}$ is singular.
\end{lemma}

\begin{proof}
First note that $I+\mathcal C_eU_j$ is nonsingular: if
$(I+\mathcal C_eU_j)v=0$, multiplication by $v^\trans U_j$ gives
$v^\trans U_jv+(U_jv)^\trans\mathcal C_e(U_jv)=0$, hence $U_jv=0$ and then $v=0$.
The same argument with $\mathcal C_{e_\ell}$ and
$\mathcal P_{e_r}+U_j$ proves that
$I+\mathcal C_{e_\ell}(\mathcal P_{e_r}+U_j)$ is nonsingular, but not
necessarily symmetric.

For any $C,Z\succeq0$,
\begin{align*}
Z(I+CZ)^{-1}
&=Z^{1/2}(I+Z^{1/2}CZ^{1/2})^{-1}Z^{1/2}\succeq0,\\
(I+CZ)^{-1}C
&=C^{1/2}(I+C^{1/2}ZC^{1/2})^{-1}C^{1/2}\succeq0.
\end{align*}
Taking $(C,Z)=(\mathcal C_e,U_j)$ in the first identity proves
$\tau(\alpha_e,U_j)-\mathcal P_e\succeq0$.  Taking
$(C,Z)=(\mathcal C_{e_\ell},\widetilde{\mathcal P}_{e_r})$ in the first and
second identities proves, respectively,
$\mathcal P_{e_{\ell r}}-\mathcal P_{e_\ell}\succeq0$ and
$\mathcal C_{e_{\ell r}}\succeq0$.  These inequalities also establish the
stated symmetry properties.

For fixed $x_i$, the left-hand side of \cref{eq:fold-identity} is the value of a
finite-dimensional convex--concave quadratic saddle problem in $(x_j,\lambda)$.
Its necessary and sufficient saddle equations are
$\lambda=U_jx_j$ and $x_j=\mathcal A_ex_i-\mathcal C_e\lambda$.  The preceding
nonsingularity gives the unique solution
$(I+\mathcal C_eU_j)x_j=\mathcal A_ex_i$; substituting it into the quadratic gives exactly the
expression in \cref{eq:terminal-fold}.  This argument uses the closed
extended-real-valued function in \cref{eq:path-value} and is therefore valid when
$\mathcal C_e$ is singular.

To prove the identity in \cref{eq:compose-identity}, introduce $\lambda_\ell$
and $\lambda_r$ for the two segment values.  For each fixed
$(x_i,x_k,\lambda_r)$, the saddle equations in the shared variables
$(x_j,\lambda_\ell)$ are
\begin{equation}
\lambda_\ell=\widetilde{\mathcal P}_{e_r}x_j
  +\mathcal A_{e_r}^\trans\lambda_r,
\qquad x_j=\mathcal A_{e_\ell}x_i
  -\mathcal C_{e_\ell}\lambda_\ell.
\label{eq:composition-stationarity}
\end{equation}
The corresponding saddle matrix is nonsingular precisely because
$I+\mathcal C_{e_\ell}\widetilde{\mathcal P}_{e_r}$ is nonsingular.  Eliminating
$(x_j,\lambda_\ell)$ and collecting
the quadratic coefficients of $(x_i,x_k,\lambda_r)$ gives the formulas in
\cref{eq:compress-intermediates,eq:composition}.  In
particular, writing $S=I+\mathcal C_{e_\ell}\widetilde{\mathcal P}_{e_r}$,
the eliminated variables
satisfy
\begin{align*}
\lambda_\ell&=(I+\widetilde{\mathcal P}_{e_r}\mathcal C_{e_\ell})^{-1}
  (\widetilde{\mathcal P}_{e_r}\mathcal A_{e_\ell}x_i
  +\mathcal A_{e_r}^\trans\lambda_r),\\
x_j&=S^{-1}(\mathcal A_{e_\ell}x_i
  -\mathcal C_{e_\ell}\mathcal A_{e_r}^\trans\lambda_r),
\end{align*}
which gives $\mathcal A_{e_{\ell r}}$ and $\mathcal C_{e_{\ell r}}$ directly;
substitution in the remaining quadratic gives $\mathcal P_{e_{\ell r}}$.
Since these equations eliminate a block
of the saddle relation itself, they remain valid for singular
$\mathcal C_{e_\ell}$ and $\mathcal C_{e_r}$, including the exact-constraint
case.  The identities
\begin{align*}
(I+\mathcal C_{e_\ell}\widetilde{\mathcal P}_{e_r})^{-1}\mathcal C_{e_\ell}
  &=\mathcal C_{e_\ell}(I+\widetilde{\mathcal P}_{e_r}\mathcal C_{e_\ell})^{-1},\\
\widetilde{\mathcal P}_{e_r}(I+\mathcal C_{e_\ell}\widetilde{\mathcal P}_{e_r})^{-1}
  &=(I+\widetilde{\mathcal P}_{e_r}\mathcal C_{e_\ell})^{-1}
  \widetilde{\mathcal P}_{e_r}
\end{align*}
are the push-through identity
$(I+AB)^{-1}A=A(I+BA)^{-1}$, applied with
$(A,B)=(\mathcal C_{e_\ell},\widetilde{\mathcal P}_{e_r})$ and
$(\widetilde{\mathcal P}_{e_r},\mathcal C_{e_\ell})$, respectively.
\end{proof}

\begin{proposition}[Contraction invariants]
\label{prop:contraction-well-posed}
Under \cref{def:assumptions}, every intermediate tree produced by contraction
satisfies
\begin{align}
U_i+\sum_{e\in\out(i)}\mathcal P_e&\succeq0
&&\text{for every node }i,\notag\\
\mathcal C_e&\succeq0
&&\text{for every edge }e.
\label{eq:contraction-invariants}
\end{align}
Consequently, if rake is applied to leaf $j$ and its parent edge $e$, then
$U_j\succeq0$ and $I+\mathcal C_eU_j$ is nonsingular.  If compress is applied
at unary node $j$, with outgoing edge $e_r$, then
$\widetilde{\mathcal P}_{e_r}=\mathcal P_{e_r}+U_j\succeq0$ and
$S=I+\mathcal C_{e_\ell}\widetilde{\mathcal P}_{e_r}$ is nonsingular.  Thus
every rake and compress update is well defined.
\end{proposition}

\begin{proof}
Initially, $U_i=Q_i$ and
$\mathcal P_e=-M_eR_e^{-1}M_e^\trans$, so the node invariant is exactly
\cref{eq:local-convexity}; $\mathcal C_e\succeq0$ follows from
\cref{eq:edge-components}.

Suppose rake removes leaf $j$ and edge $e=(i,j)$.  Because $j$ is a leaf, its
node invariant gives $U_j\succeq0$, so
\cref{lem:elimination-identities} applies.  Use a
superscript $+$ for data and outgoing-edge sets after the update.  Then
\begin{align*}
U_i^++\sum_{f\in\out^+(i)}\mathcal P_f
&=U_i+\sum_{f\in\out(i)}\mathcal P_f
  +\tau(\alpha_e,U_j)-\mathcal P_e\\
&\succeq0
\end{align*}
by \cref{eq:elimination-psd}.  The invariants at all other nodes and edges are
unchanged.  Simultaneous rakes at one parent add one positive-semidefinite
increment of this form for each removed edge.

Now suppose compress removes unary node $j$, with incoming edge $e_\ell$ from
parent $i$ and outgoing edge $e_r$.  The node invariant at $j$ gives
$U_j+\mathcal P_{e_r}\succeq0$, so
\cref{lem:elimination-identities} again applies.  The
replacement edge is $e_{\ell r}$, and
\begin{align*}
U_i^++\sum_{f\in\out^+(i)}\mathcal P_f
&=U_i+\sum_{f\in\out(i)}\mathcal P_f
  +\mathcal P_{e_{\ell r}}-\mathcal P_{e_\ell}\\
&\succeq0
\end{align*}
by \cref{eq:elimination-psd}; the other node invariants are unchanged.  The same
lemma gives $\mathcal C_{e_{\ell r}}\succeq0$, while all other edge data are
unchanged.  Induction proves \cref{eq:contraction-invariants} for every
intermediate tree.

The nonsingularity conclusions follow from the same leaf and unary cases of
\cref{eq:contraction-invariants}, together with
\cref{lem:elimination-identities}.
\end{proof}

By \cref{lem:elimination-identities,prop:contraction-well-posed}, compress composition is
associative in every intermediate tree.  For three consecutive edges with
node data $U_1,U_2$ at their two junctions,
\begin{equation}
\alpha_{e_3}\circ_{U_2}(\alpha_{e_2}\circ_{U_1}\alpha_{e_1})
=(\alpha_{e_3}\circ_{U_2}\alpha_{e_2})\circ_{U_1}\alpha_{e_1}.
\label{eq:node-aware-associativity}
\end{equation}
Both sides eliminate the same two internal states from the same
extended-real-valued conditional quadratic function.  Thus any
parenthesization of a unary chain segment, with each node's data included at
its junction, represents the same conditional problem.

\subsection{Value reconstruction algebra}

Let $\T'=(\V',\E',\rho)$ be any intermediate tree, with its current node and
edge data.  Associate with $\T'$ the control-condensed homogeneous saddle
function below, where $\lambda_e$ is the auxiliary multiplier associated with
the current edge $e$:
\begin{align}
\mathcal L_{\T'}(x,\lambda)
={}&\frac12\sum_{i\in\V'}x_i^\trans U_ix_i
  +\frac12\sum_{e=(i,j)\in\E'}x_i^\trans\mathcal P_ex_i\notag\\
&+\sum_{e=(i,j)\in\E'}
  \left\{\lambda_e^\trans(\mathcal A_ex_i-x_j)
  -\frac12\lambda_e^\trans\mathcal C_e\lambda_e\right\},
\label{eq:intermediate-saddle}
\end{align}
and its state-only value function
\begin{align}
V_{\T'}^0(x_{\V'})
&=\sup_\lambda\mathcal L_{\T'}(x,\lambda)\notag\\
&=\frac12\sum_{i\in\V'}x_i^\trans U_ix_i
 +\sum_{e=(i,j)\in\E'}\Phi_e(x_i,x_j).
\label{eq:intermediate-value}
\end{align}
For the initial tree $\T_0=\T$, this is the state function obtained from the
homogeneous version of \cref{eq:saddle} by minimizing over every control and
maximizing over every nonroot multiplier, while treating the node states as
arguments.  As with every subtree value, $V_{\T_0}^0$ excludes the
constraint entering its root.  For the full tree, this is the initial-state
saddle term involving $y_\rho$, which is incorporated when the root KKT
equations are solved.

\begin{proposition}[Partial-elimination invariant]
\label{prop:partial-elimination}
For every intermediate tree $\T'$,
\begin{equation}
V_{\T'}^0(x_{\V'})
=\inf_{x_{\V\setminus\V'}}V_{\T_0}^0(x_{\V}).
\label{eq:contraction-value-invariant}
\end{equation}
The same identity holds when $\T_0$ is any original subtree and $\T'$ is any
intermediate tree obtained by contractions within that subtree.
\end{proposition}

\begin{proof}
The identity holds trivially for $\T_0$, since
$\V\setminus\V_0=\varnothing$.  A rake preserves it by
\cref{eq:fold-identity}, and a compression preserves it by
\cref{eq:compose-identity}.  Simultaneous operations eliminate distinct state
variables whose contributions are separable conditional on the surviving
states; contributions from sibling rakes add at their common parent.  The same
argument therefore applies to an entire contraction step.  Induction over the
contraction steps proves the claim.  Starting from the initial data on any
original subtree gives
the stated subtree identity by the same induction.
\end{proof}

For an original node $i$, let $\T_i=(\V_i,\E_i,i)$ be its original subtree,
and let $V_{\T_i}^0$ be the state-only value function in
\cref{eq:intermediate-value} formed on $\T_i$ from the initial node and edge
data.  Define
\begin{equation}
V_i^0(x_i)=\inf_{x_{\V_i\setminus\{i\}}}
V_{\T_i}^0(x_{\V_i}).
\label{eq:homogeneous-subtree-value-definition}
\end{equation}
Thus $V_i^0$ is the homogeneous subtree value with $x_i$ fixed; in particular,
$V_\rho^0(x_\rho)=\inf_{x_{\V\setminus\{\rho\}}}
V_{\T_0}^0(x_\V)$.

The elimination identities also define the reverse reconstruction.  At
termination, set $P_\rho$ equal to the root data.  If rake removed leaf $j$,
set $P_j$ to the value of $U_j$ when $j$ was removed.  If compress removed $j$
between $e_\ell=(i,j)$ and $e_r=(j,k)$, then, once $P_k$ has been recovered,
set
\begin{equation}
P_j=\tau\bigl((\mathcal A_{e_r},\mathcal C_{e_r},
                 \mathcal P_{e_r}+U_j),P_k\bigr),
\label{eq:value-reconstruction}
\end{equation}
where the quantities on the right are their values when $j$ was removed.

\begin{theorem}[Value reconstruction correctness]
\label{thm:quadratic-correctness}
Under \cref{def:assumptions}, contraction and the reverse reconstruction above
are well posed.  The reconstruction returns $P_i\succeq0$ for every original
node, with
\begin{equation}
V_i^0(x_i)=\tfrac12x_i^\trans P_ix_i.
\label{eq:homogeneous-subtree-value}
\end{equation}
\end{theorem}

\begin{proof}
At termination, \cref{eq:contraction-value-invariant} reduces to
$\frac12x_\rho^\trans U_\rho x_\rho=V_\rho^0(x_\rho)$, so
$P_\rho=U_\rho$ satisfies \cref{eq:homogeneous-subtree-value}.

Proceed in reverse elimination order.  If $j$ was raked, it was a leaf of the
intermediate tree when removed.  All states strictly below $j$ in the original
tree had therefore already been eliminated.  Applying
\cref{prop:partial-elimination} to the induced contraction of the original
subtree rooted at $j$ gives
$V_j^0(x_j)=\frac12x_j^\trans U_jx_j$.  Hence restoring $P_j=U_j$ recovers the
required value.

Suppose instead that $j$ was compressed between $e_\ell=(i,j)$ and
$e_r=(j,k)$.  At removal, $e_r$ was the only outgoing edge of $j$.  Once
reverse reconstruction has recovered $P_k$, applying
\cref{prop:partial-elimination} below $j$ gives
\begin{equation}
V_j^0(x_j)=\inf_{x_k}\left\{
\tfrac12x_j^\trans U_jx_j+\Phi_{e_r}(x_j,x_k)
+\tfrac12x_k^\trans P_kx_k\right\},
\label{eq:compressed-subtree-reconstruction}
\end{equation}
where the node and edge data are their values when $j$ was removed.  Applying
\cref{eq:fold-identity} to \cref{eq:compressed-subtree-reconstruction} gives
\cref{eq:value-reconstruction} and
$V_j^0(x_j)=\frac12x_j^\trans P_jx_j$.  Reverse induction establishes
\cref{eq:homogeneous-subtree-value} for every node.

The root data are nonnegative by \cref{eq:contraction-invariants}.  A reversed
rake restores the nonnegative node data present at removal.  For a reversed
compression, the invariants at removal give
$\mathcal P_{e_r}+U_j\succeq0$ and $\mathcal C_{e_r}\succeq0$.
Given the already recovered $P_k\succeq0$,
\cref{lem:elimination-identities} shows that
\cref{eq:value-reconstruction} is well posed and returns
$P_j\succeq0$.  Reverse induction therefore gives $P_i\succeq0$ for every
node.
\end{proof}

\subsection{Factorization dataflow}

Using the node and edge initialization above, the executor applies the
following numerical callbacks.

\operation{Rake a leaf $j$ and its parent edge $(i,j)$.}
Let $e=(i,j)$.  Return the rake contribution and saved data
\begin{equation}
m_{i\leftarrow j}=\tau(\alpha_e,U_j),\qquad s_j^{\rm rake}=U_j.
\label{eq:quadratic-rake}
\end{equation}
The contribution is a symmetric matrix.  Rake contributions from siblings are
added by a balanced reduction, and their sum is added to $U_i$.  The saved data
$s_j^{\rm rake}$ are written to the expansion tape before $j$ is removed.

\operation{Compress a unary node $j$.}
Let $e_\ell$ and $e_r$ be the incoming and outgoing edges of $j$, respectively.
Return the replacement edge data and saved data
\begin{equation}
\begin{aligned}
\alpha_{e_{\ell r}}&=\alpha_{e_r}\circ_{U_j}\alpha_{e_\ell},\\
s_j^{\rm comp}&=(\mathcal A_{e_r},\mathcal C_{e_r},
\widetilde{\mathcal P}_{e_r}).
\end{aligned}
\label{eq:quadratic-compress}
\end{equation}
$s_j^{\rm comp}$ is written to the tape before the two incident edges and
$j$ are replaced by $e_{\ell r}$ and its data $\alpha_{e_{\ell r}}$.

\operation{Expand a compression at $j$.}
Once reverse expansion has recovered $P_k$ at the child endpoint, read the
saved data and set
\begin{equation}
P_j=\tau(s_j^{\rm comp},P_k).
\label{eq:quadratic-expand-compress}
\end{equation}

\operation{Expand a rake at $j$.}
Set $P_j=s_j^{\rm rake}$.

\subsection{Local factor construction}

After contraction, denote the remaining root data by $P_\rho$.  Reverse
expansion defines $P_i$ for every other original node.  The tape is temporary:
once expansion finishes, it may be discarded.  By
\cref{thm:quadratic-correctness}, these matrices are the positive-semidefinite
Hessians of the homogeneous subtree value functions.  The remaining
factorization work is node- or edge-local.  For every node, define
\begin{equation}
F_i=I+\Delta_iP_i,\qquad
W_i=P_iF_i^{-1}=(I+P_i\Delta_i)^{-1}P_i.
\label{eq:FW}
\end{equation}
For each original edge $e=(i,j)$, define
\begin{equation}
\begin{aligned}
G_e&=R_e+B_e^\trans W_jB_e,
&H_e&=B_e^\trans W_jA_e+M_e^\trans,\\
K_e&=-G_e^{-1}H_e,
&A_e^{\rm cl}&=A_e+B_eK_e,\\
T_e&=F_j^{-1}A_e^{\rm cl}.
\end{aligned}
\label{eq:factor-local}
\end{equation}
These local systems are well posed.  Indeed, if
$(I+\Delta_iP_i)v=0$, multiplication by $v^\trans P_i$ gives
$v^\trans P_iv+(P_iv)^\trans\Delta_i(P_iv)=0$, hence $P_iv=0$ and then
$v=0$.  Thus $F_i$ is nonsingular, and the push-through identity gives
\[
W_i=P_i^{1/2}
(I+P_i^{1/2}\Delta_iP_i^{1/2})^{-1}P_i^{1/2}\succeq0.
\]
It follows that $G_e=R_e+B_e^\trans W_jB_e\succ0$.

The reusable factorization consists of $(P_i,W_i)$ and a solver for $F_i$ at
each node, together with $(K_e,A_e^{\rm cl},T_e)$ and a factorization of
$G_e$ at each edge.  It contains no right-hand-side data.

\section{Rake--compress solve}
\label{sec:affine}

For fixed $(Q,M,R,A,B,\Delta)$, the KKT matrix and the factorization produced
in \cref{sec:quadratic-algebra} are unchanged when $(q,r,c)$ changes.  Each new
right-hand side therefore requires only vector-valued operations; no quadratic
coefficients are recomputed or refactored.

This section first derives a recursion for the linear coefficients of the
subtree value functions.  It then evaluates that recursion by contraction and
reverse expansion, solves for the root state, and recovers the remaining
states, controls, and multipliers.  The final subsection states the associated
executor dataflow.

\subsection{Affine subtree recursion}

The solve uses separate affine data rather than appending vectors to the
quadratic edge triples.  For arbitrary $(q,r,c)$, let $V_i(x_i)$ be the value
of the nonhomogeneous problem restricted to the original subtree rooted at
$i$, with $x_i$ fixed and the constraint entering $i$ omitted.  Write
\[
V_i(x_i)=\tfrac12x_i^\trans P_ix_i+p_i^\trans x_i+\kappa_i,
\]
where $P_i$ is provided by the factorization and $p_i$ is the linear
coefficient to be computed.  For each original edge $e=(i,j)$, form
\begin{equation}
Z_e=T_e^\trans,\qquad
z_e=K_e^\trans r_e+(A_e^{\rm cl})^\trans W_jc_j.
\label{eq:affine-edge}
\end{equation}
For later use, define
\begin{equation}
\begin{aligned}
f_i&=\Delta_ip_i-c_i,\\
g_i&=p_i-W_if_i,\\
k_e&=-G_e^{-1}(r_e+B_e^\trans g_j),\\
b_e&=F_j^{-1}(B_ek_e-f_j),\qquad e=(i,j).
\end{aligned}
\label{eq:local-affine-solve}
\end{equation}

\begin{lemma}[Affine subtree recursion]
\label{lem:affine-recursion}
Under \cref{def:assumptions}, fix the factorization produced in
\cref{sec:quadratic-algebra} and arbitrary affine data $(q,r,c)$.  The linear
coefficients of the nonhomogeneous subtree values satisfy
\begin{equation}
p_i=q_i+\sum_{e=(i,j)\in\out(i)}(Z_ep_j+z_e).
\label{eq:affine-message}
\end{equation}
\end{lemma}

\begin{proof}
After the variables strictly below $j$ have been eliminated, the terms in the
parent problem that depend on $x_j$ are
$V_j(x_j)-y_j^\trans x_j$.  Stationarity with respect to $x_j$ therefore gives
\[
0=\nabla_{x_j}\bigl(V_j(x_j)-y_j^\trans x_j\bigr)
  =P_jx_j+p_j-y_j,
\]
and hence $y_j=P_jx_j+p_j$.
Now fix $e=(i,j)$.  Substituting this relation into the child dynamics
\cref{eq:kkt-dyn} gives
\[
(I+\Delta_jP_j)x_j
=A_ex_i+B_eu_e-(\Delta_jp_j-c_j).
\]
Using \cref{eq:FW,eq:local-affine-solve}, solve for $x_j$ and substitute the
result into $y_j=P_jx_j+p_j$ to obtain
\[
\begin{aligned}
x_j&=F_j^{-1}(A_ex_i+B_eu_e-f_j),\\
y_j&=W_j(A_ex_i+B_eu_e)+g_j.
\end{aligned}
\]
Substituting these expressions into \cref{eq:kkt-u} and using
\cref{eq:factor-local} yields
$G_eu_e+H_ex_i+r_e+B_e^\trans g_j=0$.  Equations
\eqref{eq:factor-local} and \eqref{eq:local-affine-solve} then give
\begin{equation}
\begin{aligned}
u_e&=K_ex_i+k_e,\\
y_j&=W_jA_e^{\rm cl}x_i+W_jB_ek_e+g_j,
\end{aligned}
\label{eq:eliminated-edge-affine}
\end{equation}
where the second identity follows by substituting the first into the preceding
expression for $y_j$.  Hence the part of
$M_eu_e+A_e^\trans y_j$ independent of $x_i$ is
\[
d_e=M_ek_e+A_e^\trans(W_jB_ek_e+g_j).
\]
Since $G_e$ is symmetric, transposing $G_eK_e=-H_e$ from
\cref{eq:factor-local} yields
$M_e+A_e^\trans W_jB_e=-K_e^\trans G_e$.  Together with
\cref{eq:local-affine-solve}, this gives
\[
d_e=K_e^\trans r_e+(A_e^{\rm cl})^\trans g_j.
\]
The push-through identity and \cref{eq:FW} give
\[
g_j=(I-W_j\Delta_j)p_j+W_jc_j
    =F_j^{-\trans}p_j+W_jc_j.
\]
Using \cref{eq:factor-local,eq:affine-edge}, it follows that
$d_e=T_e^\trans p_j+z_e=Z_ep_j+z_e$.  In the homogeneous case,
\cref{thm:quadratic-correctness} gives $y_i=P_ix_i$, while
\cref{eq:eliminated-edge-affine} reduces to
$u_e=K_ex_i$ and $y_j=W_jA_e^{\rm cl}x_i$.  Substitution into
\cref{eq:kkt-x} therefore gives
\[
P_i=Q_i+\sum_{e=(i,j)\in\out(i)}
\left(M_eK_e+A_e^\trans W_jA_e^{\rm cl}\right).
\]
For arbitrary affine data, substituting $y_i=P_ix_i+p_i$ and
\cref{eq:eliminated-edge-affine} into \cref{eq:kkt-x}, and using this identity
to cancel the coefficient of $x_i$, leaves
$p_i=q_i+\sum_{e=(i,j)\in\out(i)}d_e$.  This is
\cref{eq:affine-message}.
\end{proof}

\subsection{Affine contraction and expansion}

Associate with every intermediate tree $\T'$ node vectors $h_i$ and affine
edge maps $\zeta_e(v)=Z_ev+z_e$.  These data define vectors
$\pi_i^{\T'}$ recursively by
\begin{equation}
\pi_i^{\T'}=h_i+
\sum_{e=(i,j)\in\out(i)}\zeta_e(\pi_j^{\T'}).
\label{eq:intermediate-affine-recursion}
\end{equation}
Initially, $h_i=q_i$ and the edge maps are given by
\cref{eq:affine-edge}, so \cref{lem:affine-recursion} gives
$\pi_i^{\T_0}=p_i$.

A rake of leaf $j$ through the edge map $\zeta_e(v)=Zv+z$ adds
\begin{equation}
m_{i\leftarrow j}=Zh_j+z
\label{eq:affine-rake}
\end{equation}
to its parent data.  For a compression, let $(Z_\ell,z_\ell)$ be the
parent-to-middle map, $(Z_r,z_r)$ the middle-to-child map, and $h_j$ the
middle-node data.  Substitution gives the replacement map
\begin{equation}
\widehat\zeta(v)=Z_\ell Z_rv+z_\ell+Z_\ell(h_j+z_r).
\label{eq:affine-compose}
\end{equation}
The inverse relations are $p_j=h_j$ for a rake and
$p_j=Z_rp_k+h_j+z_r$ for a compression once $p_k$ is known.

\begin{proposition}[Affine evaluation invariant]
\label{prop:affine-evaluation}
Affine contraction preserves the vectors defined by
\cref{eq:intermediate-affine-recursion} at every surviving node.
Consequently, contraction produces $p_\rho$, and reverse expansion produces
$p_i$ for every original node.
\end{proposition}

\begin{proof}
For a rake, \cref{eq:affine-rake} is exactly the term removed from the
parent's sum in \cref{eq:intermediate-affine-recursion}; sibling contributions
add.  For a compression, substituting
$\pi_j^{\T'}=h_j+Z_r\pi_k^{\T'}+z_r$ into
$Z_\ell\pi_j^{\T'}+z_\ell$ gives \cref{eq:affine-compose}.  Thus every
operation preserves the recursively defined vector at each surviving node,
and the root data at termination equal $p_\rho$.  The inverse relations above
then recover every $p_i$ by reverse induction.
\end{proof}

Once all $p_i$ have been recovered, the quantities in
\cref{eq:local-affine-solve} are evaluated, and the root state is
\begin{equation}
x_\rho=-F_\rho^{-1}f_\rho.
\label{eq:xroot}
\end{equation}
For every original edge, define the parent-to-child state map
\begin{equation}
\psi_e(x)=T_ex+b_e.
\label{eq:state-map}
\end{equation}
Two consecutive maps compose according to
\begin{equation}
(T_r,b_r)\circ(T_\ell,b_\ell)
=(T_rT_\ell,T_rb_\ell+b_r).
\label{eq:transition-compose}
\end{equation}
Composition preserves the parent-to-descendant map represented between the
surviving endpoints.  Thus contraction produces the exact composed map for
each contracted chain segment, while the saved constituent map reconstructs
the eliminated node during reverse expansion.
An ordinary root-to-leaf traversal would have span equal to the tree height.
Instead, the executor contracts the state maps according to the same topology
plan, composing adjacent maps by \cref{eq:transition-compose} and saving the
map needed to restore each removed node.  Starting from $x_\rho$, reverse
expansion then evaluates all maps in each reverse dependency level
concurrently.  Hence state recovery depends on the depth of the contraction
plan's primitive dependency graph rather than the height of the original tree;
\cref{sec:complexity} proves that this depth is $O(\log N)$.  At a rake,
several children may read the same parent state concurrently, while each child
state has a unique writer.

The controls and multipliers then follow independently from
\begin{align}
x_j&=T_ex_i+b_e,&u_e&=K_ex_i+k_e,\qquad e=(i,j),\notag\\
y_i&=P_ix_i+p_i.
\label{eq:recover}
\end{align}

\begin{theorem}[Solve correctness]
\label{thm:solve-correctness}
Under \cref{def:assumptions}, for every right-hand side $(q,r,c)$, the
rake--compress solve described in this section returns the unique solution of
the KKT system in \cref{eq:kkt}.
\end{theorem}

\begin{proof}
By \cref{lem:affine-recursion,prop:affine-evaluation}, the recovered $p_i$ are
the linear coefficients of the reduced subtree values.  Exact composition in
\cref{eq:transition-compose} ensures that the reconstructed states satisfy the
state relation $x_j=T_ex_i+b_e$ in \cref{eq:recover}.  The same equation sets
$u_e=K_ex_i+k_e$ and $y_i=P_ix_i+p_i$.  We verify that these quantities
satisfy \cref{eq:kkt}.

\Cref{thm:quadratic-correctness,lem:affine-recursion} identify
$V_i(x_i)=\tfrac12x_i^\trans P_ix_i+p_i^\trans x_i+\kappa_i$ as the
reduced subtree value.  Because the eliminated variables satisfy their
stationarity equations, the envelope identity for this partial saddle
elimination gives
\[
\nabla_{x_i}V_i(x_i)=Q_ix_i+q_i+
\sum_{e=(i,j)\in\out(i)}(M_eu_e+A_e^\trans y_j).
\]
Since \cref{eq:recover} sets
$y_i=P_ix_i+p_i=\nabla_{x_i}V_i(x_i)$, this identity is
\cref{eq:kkt-x}.  For $e=(i,j)$, the definitions of $T_e$ and $b_e$ give
\[
F_jx_j=A_ex_i+B_eu_e-f_j.
\]
Together with \cref{eq:FW,eq:local-affine-solve} and
$y_j=P_jx_j+p_j$, this proves \cref{eq:kkt-dyn}.  The same equations yield
$y_j=W_j(A_ex_i+B_eu_e)+g_j$, so the left side of
\cref{eq:kkt-u} is
\[
(G_eK_e+H_e)x_i+G_ek_e+r_e+B_e^\trans g_j=0
\]
by \cref{eq:factor-local,eq:local-affine-solve}.  Finally,
$F_\rho x_\rho=-f_\rho$ is equivalent to \cref{eq:kkt-root}.  Thus all KKT
equations hold, and uniqueness follows from \cref{prop:unique}.
\end{proof}

\subsection{Solve dataflow}

The executor first evaluates the affine recursion.  Node $i$ is initialized
with $h_i=q_i$, and edge $e$ with $\zeta_e=(Z_e,z_e)$.  Its callbacks are:
\operation{Rake.}
Return the rake contribution in \cref{eq:affine-rake}, save $s_j^p=h_j$, and
add sibling contributions to the parent data by a balanced reduction.

\operation{Compress.}
Replace the two incident edge maps by \cref{eq:affine-compose} and save
$s_j^p=(Z_r,h_j+z_r)$.

\operation{Expansion.}
For a compression, use $s_j^p$ and the recovered $p_k$ to set
$p_j=Z_rp_k+h_j+z_r$.  For a rake, set $p_j=s_j^p$.
The resulting affine tape is specific to the current right-hand side and may
be discarded once every $p_i$ has been recovered.

The executor then applies the same contraction plan to the state maps in
\cref{eq:state-map}.  This pass is a broadcast rather than an upward
recursion, so a rake contributes the additive identity to the parent data.
Its callbacks are:
\operation{Rake.}
For $e=(i,j)$, return a zero rake contribution and save
$s_j^x=(T_e,b_e)$.

\operation{Compress.}
For $e_\ell=(i,j)$ and $e_r=(j,k)$, replace their maps by the composition in
\cref{eq:transition-compose} and save the parent-to-middle map
$s_j^x=(T_\ell,b_\ell)$.

\operation{Expansion.}
Once $x_i$ has been recovered, restore a compressed middle node from
$s_j^x$ as $x_j=T_\ell x_i+b_\ell$, or a raked leaf as
$x_j=T_ex_i+b_e$.
After \cref{eq:recover} has been evaluated, the state tape may be discarded.

\section{Complete algorithm}
\label{sec:algorithm}

\Cref{alg:solver} combines the three algebraic passes.  The topology plan is
constructed once and reused.  A factorization is recomputed only when the KKT
matrix changes; any number of right-hand-side solves may then reuse it.

\begin{center}
\refstepcounter{algorithm}
\label{alg:solver}
\fbox{\begin{minipage}{0.95\columnwidth}
\small
\textbf{Algorithm \thealgorithm: parallel branched dual-regularized LQR.}
\begin{enumerate}
\item Build a topology-only contraction plan from the parent array (once).
\item Eliminate every edge control locally from the homogeneous quadratic data
to form the edge triples in \cref{eq:edge-components}; rake and compress these
triples, then expand the quadratic algebra to obtain all $P_i$.
\item Construct reusable solvers for $F_i$ and $G_e$ and form
$W_i,K_e,A_e^{\rm cl}=A_e+B_eK_e,T_e$ locally in parallel.
\item Form the original-edge affine maps $(Z_e,z_e)$; rake and compress these
maps, then expand the affine algebra to obtain all $p_i$.
\item Form $f_i,g_i,k_e,b_e$ and solve the root equation locally.
\item Contract the transition maps and expand from $x_\rho$ to obtain all $x_i$.
Compute $u_e$ and $y_i$ independently using the formulas in
\cref{eq:recover}.
\end{enumerate}
\end{minipage}}
\smallskip
\noindent\footnotesize The topology plan is reusable across all calls with the
same tree topology.
Steps 2--3 factor the KKT matrix; Steps 4--6 solve one right-hand side.
\end{center}

\begin{theorem}[End-to-end correctness]
\label{thm:end-to-end}
Under \cref{def:assumptions}, Algorithm~\ref{alg:solver} returns the unique
solution of the KKT equations in \cref{eq:kkt}.
\end{theorem}

\begin{proof}
\Cref{thm:quadratic-correctness} proves that Steps 2--3 produce the exact
Hessians of the homogeneous subtree value functions and well-posed local
factors.
\Cref{thm:solve-correctness} proves that Steps 4--6 use those factors to solve
the KKT system for an arbitrary right-hand side.
\end{proof}

\section{Parallel complexity}
\label{sec:complexity}

\subsection{Deterministic structural schedule}

The contraction schedule depends only on the parent array.  At the start of a
structural round, rake every nonroot leaf.  After those removals, select
a maximal set of pairwise nonadjacent unary middle nodes and compress them in
parallel.  The implementation constructs this set greedily in fixed node-index
order, selecting an eligible center unless its parent or child has already
been selected; the analysis below applies to any maximal selection.  The
eliminated centers and the pairs of adjacent edges they replace are disjoint,
although surviving endpoint nodes may be shared by several compressions.
Rake contributions from siblings in the same round are combined by a precomputed
order-preserving reduction.  Its parenthesization accounts for the levels at
which its inputs become available, so no two operations write the same parent
and an already late input is not placed at the bottom of an otherwise balanced
reduction.  These rounds define the primitive operation set; they need not
become global execution barriers.  If executed
synchronously, expansion visits rounds backward and undoes compressions before
rakes because the forward order was rake then compress.

\Cref{fig:example-schedules} introduced representative schedules at the
topology level.  We now quantify their progress and refine their structural
rounds into a barrier-free numerical dependency schedule.

\begin{lemma}[Geometric progress]
\label{lem:progress}
Every nontrivial round of the stated scheduler removes at least one third of the
nodes present at its beginning, rounded up.  Consequently it has at most
$\lceil\log_{3/2}|\V|\rceil$ rounds.  The one-third removal bound is tight.
\end{lemma}

\begin{proof}
Write $N$ for the number of nodes at the beginning of the round, and let $L$
be the number of leaves raked.  If raking leaves only the root, then
$L=N-1\ge N/2$.
Otherwise, in the post-rake tree,
let $L',U',B'$ count leaves, unary nodes, and nodes with at least two children.
Every new leaf had at least one distinct child among the raked leaves, hence
$L'\le L$.  The rooted-tree degree identity gives $B'\le L'-1$ when more than
one node remains.

Every nonroot unary node is an eligible compression center, and at most one
unary node---the root---is ineligible.  Let $\epsilon\in\{0,1\}$ indicate
whether the root is unary.  The eligible centers induce disjoint paths.  By
maximality, every eligible center is selected or is adjacent to a selected
center, while one selected center accounts for at most itself and its two
neighbors.  Thus, for the selected set $S$,
\begin{equation}
U'-\epsilon\le3|S|.
\label{eq:selected-bound}
\end{equation}
Also
\[
N=L+L'+U'+B'\le3L+U'-1.
\]
Therefore
\[
3(L+|S|)\ge3L+U'-\epsilon\ge N,
\]
so the round removes at least $\lceil N/3\rceil$ nodes.  Iterating the resulting
$2/3$ upper bound on the number of remaining nodes proves the round bound.

For tightness, take a unary root above a full binary tree with $k$ leaves and
attach one additional child to each of those leaves.  The resulting tree has
$1+(k-1)+k+k=3k$ nodes.  The round rakes the $k$ added children; afterward the
only unary node is the ineligible root, so there are no compressions.  Exactly
one third of the nodes are removed.
\end{proof}

\subsection{Static dependency levels}

The structural rounds determine which nodes are eliminated, but do not by
themselves specify the dependencies among the corresponding numerical
operations.  We make these dependencies explicit by decomposing each round
into four types of primitive operation.  Evaluating a rake produces one
contribution to the data of the leaf's parent.  Contributions associated with
one parent's children belong to a set equipped with a closed associative
binary operation $\mathbin{\oplus}$.  A binary combination applies
$\mathbin{\oplus}$ to the aggregates of two consecutive sibling intervals,
producing the aggregate for their union.  Once all contributions to one parent
have been combined, an application-defined absorption updates that parent's
node data using the resulting aggregate.  A compression produces the data of
the edge that replaces its two incident edges.  The scheduler requires each
rake, combination, absorption, and compression to have bounded arity; its span
bound counts the application-specific cost of these primitive operations.
Because the reduction preserves sibling order, associativity is sufficient:
commutativity is not required.

In the LQR instantiations of \cref{sec:quadratic-algebra,sec:affine},
$\mathbin{\oplus}$ is addition.  For example, in the factorization pass a rake
of $j$ produces $m_{i\leftarrow j}$ from \cref{eq:quadratic-rake}, and
absorption performs $U_i\leftarrow
U_i+\sum_jm_{i\leftarrow j}$.  The affine solve pass uses the same pattern with
its affine rake contributions.  In state recovery the rake contribution is
the additive identity because no sibling data must be accumulated.

These operations and their data dependencies form a directed acyclic graph.
Every initial node or edge datum has \emph{producer level} zero.  Recursively,
an operation whose input data have producer levels
$\ell_1,\ldots,\ell_s$ is assigned \emph{dependency level}
\begin{equation}
1+\max_{1\leq j\leq s}\ell_j,
\label{eq:earliest-level}
\end{equation}
and each output datum receives the level of the operation that produces it;
we denote the producer level of datum $z$ by $\ell(z)$.
Thus dependency level $t$ contains exactly the operations that can start after
$t-1$ successive layers of dependencies have completed; equivalently, $t$ is
the number of operations on a longest dependency chain ending at that
operation.  Operations in one level may read common data, but every output
datum has a unique producer.  The
intermediate aggregates of sibling contributions remain available until their
final aggregate
has been absorbed.  This earliest-start layering does not impose barriers at
the boundaries of structural rounds: an operation from a later round may begin
as soon as its own inputs are available.

Consider the $k$ rake contributions $c_1,\ldots,c_k$ to one parent, in sibling
order, and let $\ell_j$ be the producer level of $c_j$.  Their reduction must
choose a binary parenthesization.  Equivalently, it must choose a binary tree
whose ordered leaves are the contributions and whose internal vertices apply
$\mathbin{\oplus}$.  If leaf $j$ has depth $d_j$, its contribution passes
through $d_j$ combinations, so its path to the final aggregate is ready at
level $\ell_j+d_j$.  A reduction tree obtained by recursively balancing the
numbers of contributions in its two subtrees controls $d_j$ but ignores the
producer levels $\ell_j$; in particular, it may place an already late
contribution below several avoidable combinations.  We instead choose the
reduction tree to control $\max_j(\ell_j+d_j)$ while preserving sibling order.
This is an alphabetic minimax-tree objective~\cite{kirkpatrick1985alphabetic};
nonuniform input availability plays an analogous role in the design of
parallel-prefix circuits~\cite{held2017prefix}.

The construction assigns the readiness weight $w_j=2^{\ell_j}$ and defines
\begin{equation}
\begin{aligned}
W&=\sum_{j=1}^kw_j, &
p_j&=\frac{w_j}{W},\\
F_j&=\sum_{q<j}p_q, &
d_j^0&=\left\lceil\log_2\frac{1}{p_j}\right\rceil+1.
\end{aligned}
\label{eq:readiness-code-data}
\end{equation}
For contribution $j$, let $m_j=F_j+p_j/2$ and take the first $d_j^0$ binary
digits of this midpoint.  These finite words are the alphabetic Shannon--Fano--Elias
code~\cite[Sec.~5.9]{cover2006elements}.  Their \emph{binary trie} is the
rooted binary tree in which words with a common prefix share the corresponding
initial path.  Use its codewords as the leaves, contract every vertex with only
one child, and associate each remaining internal vertex with the combination
of its two child aggregates.  The next lemma shows that the codewords are
prefix-free and ordered, so this procedure indeed produces an order-preserving
binary reduction tree.  It also shows why the exponential weights are useful:
the shorter path assigned to a contribution with larger $\ell_j$ offsets its
later availability.

\begin{lemma}[Readiness-weighted reduction]
\label{lem:weighted-reduction}
For any associative operation $\mathbin{\oplus}$, the reduction above evaluates
$c_1\mathbin{\oplus}\cdots\mathbin{\oplus}c_k$ using a parenthesization that
preserves the original sibling order, uses $k-1$ combinations, and produces
its final aggregate by level
\begin{equation}
\ell_{\rm out}\le
\left\lceil\log_2 W\right\rceil+1.
\label{eq:weighted-reduction-level}
\end{equation}
\end{lemma}

\begin{proof}
Let $r_j$ be the integer whose length-$d_j^0$ binary representation is the word
for contribution $j$.  The word identifies the dyadic interval
$I_j=[r_j2^{-d_j^0},(r_j+1)2^{-d_j^0})$.  Because the first $d_j^0$ bits of
the binary expansion of $m_j$
(using the expansion with trailing zeros when $m_j$ is dyadic) represent
$r_j$, the remaining bits give a number $\theta_j\in[0,1)$ such that
$m_j=(r_j+\theta_j)2^{-d_j^0}$.  Hence $m_j\in I_j$.  The interval has width
$2^{-d_j^0}\le p_j/2$.  Since $m_j\in I_j$, every point of $I_j$ lies within
$p_j/2$ of $m_j$.  As $m_j$ is the midpoint of
$[F_j,F_j+p_j]$, it follows that $I_j\subseteq[F_j,F_j+p_j)$.  The intervals
$[F_j,F_j+p_j)$ are disjoint and occur in
sibling order, and hence so do the intervals $I_j$.  If one codeword were a
prefix of another, the dyadic interval of the former would contain that of the
latter, contradicting their disjointness.  The codewords are therefore
prefix-free.  Moreover, disjoint dyadic intervals occur from left to right in
the same order as their binary words occur lexicographically.  Thus the
lexicographic codeword order is the original sibling order.

Place the $0$-child to the left of the $1$-child at every trie vertex.  The
leaves then appear as $c_1,\ldots,c_k$ from left to right, and every subtree
contains a consecutive interval of contributions.  Suppressing unary
vertices preserves this order and cannot increase any leaf depth.  The
resulting tree is a full binary tree with $k$ leaves, so it has exactly $k-1$
internal vertices and therefore specifies exactly $k-1$ combinations.  It
changes only the parenthesization of
$c_1\mathbin{\oplus}\cdots\mathbin{\oplus}c_k$; associativity therefore gives
the stated aggregate.

Let $d_j$ be the depth of leaf $j$ in this reduction tree.  Each internal
combination adds one dependency level, so the level of the final aggregate is
\[
\ell_{\rm out}=\max_j(\ell_j+d_j).
\]
Unary suppression gives
$d_j\le d_j^0=\lceil\log_2(W/w_j)\rceil+1$.  Since
$\ell_j=\log_2w_j$ is an integer,
\[
\ell_j+d_j
\le \ell_j+\left\lceil\log_2W-\ell_j\right\rceil+1
=\lceil\log_2W\rceil+1.
\]
Taking the maximum over $j$ proves \cref{eq:weighted-reduction-level}.
\end{proof}

\begin{lemma}[Arbitrary-degree primitive depth]
\label{lem:async-depth}
For every $N$-node rooted tree, the dependency layering defined by
\cref{eq:earliest-level} has at most
\begin{equation}
D_N=\left\lfloor
\log_2(2N-1)
+\left\lceil\log_{3/2}N\right\rceil\log_2 36
\right\rfloor
\label{eq:dependency-level-bound}
\end{equation}
nonempty levels, without a bound on out-degree.  Its reverse recovery graph
has at most $D_N$ levels as well.  In particular, both depths are $O(\log N)$.
\end{lemma}

\begin{proof}
Set $\omega(z)=2^{\ell(z)}$.  Since $\ell(z)=\log_2\omega(z)$, it suffices to
bound these readiness weights.  Let $\Omega_r$ be their sum over the node and
edge data present at the beginning of structural round $r$.
Initially, the $N$ node-data bundles and $N-1$ edge-data bundles all have
producer level zero, so
\[
\Omega_1=(N+(N-1))2^0=2N-1.
\]
Suppose a primitive takes data $z_1,\ldots,z_s$ and produces $y$.  By
\cref{eq:earliest-level},
\begin{equation}
\omega(y)=2^{1+\max_j\ell(z_j)}
\le2\sum_{j=1}^s2^{\ell(z_j)}
=2\sum_{j=1}^s\omega(z_j).
\label{eq:local-readiness-weight}
\end{equation}
Consequently, if pairwise input-disjoint primitives consume data of total
weight $S$, their outputs have total weight at most $2S$.  Replacing those
inputs within a collection of weight $T$ therefore leaves weight at most
$T-S+2S\le2T$.  We apply this replacement bound to the four operation groups
of round $r$.

\emph{Rakes.}
The rakes consume disjoint leaf-node and edge data.  Let $S\le\Omega_r$ be
their total input weight.  Equation~\eqref{eq:local-readiness-weight} bounds
the total weight of their contributions by $2S\le2\Omega_r$; the node and edge
data not consumed by the rakes have weight $\Omega_r-S\le\Omega_r$.

\emph{Sibling reductions.}
For each parent $i$ receiving at least one contribution, let $\mathcal C_i$ be
its rake contributions, let
$\Gamma_i=\sum_{c\in\mathcal C_i}\omega(c)$, and let $a_i$ be their final
aggregate.  In \cref{lem:weighted-reduction}, $W=\Gamma_i$, and hence
\begin{equation}
\omega(a_i)=2^{\ell(a_i)}
\le 2^{\lceil\log_2\Gamma_i\rceil+1}
\le4\Gamma_i.
\label{eq:aggregate-readiness-weight}
\end{equation}
The ceiling and the additional level in \cref{lem:weighted-reduction} account
for the two factors of two in this bound.
Every contribution belongs to exactly one parent, so
$\sum_i\Gamma_i\le2\Omega_r$.  The final aggregates therefore have total
weight at most $8\Omega_r$.  Together with the unconsumed node and edge data,
their total weight before absorption is at most $9\Omega_r$.  Every
intermediate aggregate lies on a dependency path to its parent's final
aggregate; its producer level and weight are therefore no greater than those
of that final aggregate, and in particular its weight is at most
$8\Omega_r$.

\emph{Absorptions.}
Each absorption replaces one parent-node datum and its final aggregate.
Different absorptions have disjoint inputs, so the replacement bound gives a
total weight of at most $18\Omega_r$ after all absorptions.

\emph{Compressions.}
The selected compressions have disjoint middle-node and edge inputs.  They may
share surviving endpoints, but the data of those endpoints are not inputs.
One more application of the replacement bound gives
\begin{equation}
\Omega_{r+1}\le36\Omega_r.
\label{eq:round-readiness-growth}
\end{equation}
Rake, reduction, absorption, and compression outputs have weight at most
$2\Omega_r$, $8\Omega_r$, $18\Omega_r$, and $36\Omega_r$, respectively.

By \cref{lem:progress}, the number of structural rounds satisfies
$R\le\lceil\log_{3/2}N\rceil$.  Equations~\eqref{eq:round-readiness-growth}
and the initial value give
$\Omega_{r}\le(2N-1)36^{r-1}$.  The within-round bound above therefore gives,
for every datum $z$ produced in any round,
\[
2^{\ell(z)}\le(2N-1)36^R.
\]
Taking base-two logarithms and using the bound on $R$ yields
$\ell(z)\le D_N$.  Equation~\eqref{eq:earliest-level} assigns exactly these
producer levels.  Moreover, an operation at level $t$ lies at the end of a
dependency chain containing operations at levels $1,\ldots,t$, so the number
of nonempty levels is at most $D_N$.

Reverse recovery traverses the recorded elimination dependencies backward.
Each expansion callback corresponds to one recorded rake or compression and
depends only on boundary data restored by callbacks later in forward
dependency order.  Thus every chain of expansion callbacks maps to the reverse
of a chain in the forward elimination graph.  Sibling-combination operations
have no reverse callback, so the reverse callback graph has at most $D_N$
levels.  Each callback has constant numerical depth; hence reverse recovery
also has $O(\log N)$ span.
\end{proof}

\begin{theorem}[Work, storage, and span]
\label{thm:complexity}
Let $N=|\V|$, let $\chi(n,m)$ bound the work of one dense local LQR
operation, and let $\delta(n,m)$ bound its span.  Given a contraction plan, the
numerical stages 2--6 of Algorithm~\ref{alg:solver} use
\begin{equation}
O(N\chi(n,m))\ \text{work},\qquad
O(N(n^2+nm+m^2))\ \text{storage}.
\end{equation}
Their span is
\begin{equation}
O\!\left(\delta(n,m)\log N\right).
\label{eq:span}
\end{equation}
In particular, for fixed block dimensions the work and span are $O(N)$ and
$O(\log N)$, respectively, independently of tree height, balance, and maximum
out-degree.
\end{theorem}

\begin{proof}
Every nonroot node is removed exactly once by rake or compress.  A reduction of
$k$ sibling contributions uses $k-1$ binary combinations, so the number of
local algebra operations
over all rounds is $O(N)$.  Reverse expansion restores every removed node once.
The dimensions of all node, edge, and tape records depend on $n$ and $m$, but
not on $N$.  Dependency-level execution additionally retains one rake
contribution per rake until it has been combined and absorbed, so storage
remains linear.  By \cref{lem:async-depth},
each contraction and expansion pass has $O(\delta(n,m)\log N)$ span.  Algorithm
\ref{alg:solver} uses a constant number of these passes and constant-depth local
factorization stages, which preserves the span bound in \cref{eq:span}.
\end{proof}

\begin{corollary}[Native scenario-MPC form]
\label{cor:native-complexity}
Under the assumptions in \cref{eq:native-assumptions} and for fixed native
state and control dimensions, the node-control problem in
\cref{eq:node-lagrangian} on an arbitrary $N$-node scenario tree has an exact
$O(N)$-work, $O(N)$-storage, and $O(\log N)$-span numerical solve, given a
topology plan, by applying the lifting in \cref{prop:lifting} followed by
Algorithm~\ref{alg:solver}.
\end{corollary}

\begin{proof}
The lifting has at most $2N-1$ nodes and replaces $(n,m)$ by the fixed dimensions
$(n+m,m)$.  Apply
Theorems~\ref{thm:end-to-end} and~\ref{thm:complexity}, then project using
\cref{prop:lifting}.
\end{proof}

The theorem assumes that a contraction plan is available.  The current host
planner scans all node slots in every round and therefore has
$O(N\log N)$ worst-case preprocessing work.  The plan is reused across MPC
iterations whenever the parent array is unchanged.  Constructing dependency
levels is linear in the number of planned operations.  The timing study
therefore treats planning and compilation as reusable setup and reports the
post-compilation factorization and right-hand-side solve separately.

\section{JAX implementation}
\label{sec:implementation}

The rake--compress package~\cite{rakecompressjax} accepts one parent
index per node.  Its \texttt{TreeContractionPlan} stores a static schedule of
rakes, sibling reductions, absorptions, and compressions.  Its dependency depth
satisfies the bound in \cref{lem:async-depth}.  The schedule contains separate
integer arrays
for rakes, sibling reductions, absorptions, and compressions; retaining each
rake contribution until its absorption lets these operations overlap work
elsewhere.  Numerical data
are arbitrary JAX PyTrees with a leading node or edge axis.  Local algebra
callbacks are vectorized with \texttt{jax.vmap}; indexed updates have disjoint
destinations and therefore require no conflict resolution at the algorithmic
level.  \texttt{tree\_contract} returns root data
and a residual tape, and \texttt{tree\_expand} applies algebra-specific reverse
callbacks.  A Python loop over the $O(\log N)$ scheduled operation groups is
unrolled during tracing, while work within each group remains batched.  The
same code composes with JIT compilation, vectorization, and automatic
differentiation in JAX~\cite{jax2018}.

The consumer package \path{regularized_lqr_jax}~\cite{regularizedlqrjax}
implements Algorithm~\ref{alg:solver}.  Node arrays retain caller order; edge
arrays follow the plan's child ordering.  The \texttt{factor\_tree} and
\texttt{solve\_tree} calls share a plan-selected executor, so one factorization
may serve several affine right-hand sides.

In turn, \texttt{primal-dual-lipa}~\cite{primalduallipa} consumes this solver
and exposes a primal--dual interior-point method for nonlinear tree-structured
optimal-control problems in the
edge-control normal form of \cref{sec:problem}, with sequential and parallel
tree Riccati solves.  Users provide node- and edge-local costs and constraints,
edge dynamics, and the parent array through a common interface; selecting
sequential or parallel execution changes only the tree-LQR kernel.  The package
uses automatic differentiation to form the Newton subproblems and handles the
barrier terms, line search, and recovery of the complete primal--dual solution.
The contraction plan and compiled solver are reused while the topology and
static problem structure remain unchanged.

The current JAX implementation uses uniform array shapes so that all operations
of one kind in a level can share compiled dense kernels.  This is an
implementation choice, not a restriction of the algorithm: other backends may
use dynamically sized or shape-polymorphic kernels, while static-shape backends
may use padding, batches of equal-shaped operations, or separate compiled
plans.  JAX/XLA may lower one logical level to multiple device kernels, so the
PRAM-style span is an algorithmic dependency bound, not a promise of one
hardware dispatch per level.  Accelerator speed also depends on batching
occupancy, memory traffic, fusion, compilation caching, and the distribution
of operation counts across irregular levels.

\section{Numerical validation}
\label{sec:validation}

\subsection{KKT accuracy}

We assembled the dense matrix for the KKT system in \cref{eq:kkt} independently
in NumPy and compared its solution with the $(x,u,y)$ returned by
Algorithm~\ref{alg:solver}.  The test set
contains 26 instances in IEEE~754 double precision (binary64): ten fixed chain,
star, and irregular cases with both $\Delta=0$ and $\Delta\succeq0$, and sixteen
randomly generated trees whose node labels were permuted to destroy topological
order.  State dimensions ranged from one to four and control dimensions from one
to three.  With JAX 0.10.1 on the CPU, the maximum absolute KKT residual was
$3.89\times10^{-15}$, and the maximum absolute difference from the independent
dense solution over $(x,u,y)$ was $1.25\times10^{-14}$.  Additional tests check
chain-executor equivalence, plan selection, factor reuse, finite
single-precision residuals, and differentiation through a right-hand side.

\subsection{Contraction depth}

\Cref{tab:depths} reports the primitive operation levels measured for four
topology families.  ``Balanced'' is a complete binary tree; ``comb'' alternates a
trunk with side leaves.  In every case the abstract number of rakes plus
compressions is $N-1$, as required for linear numerical work.  Chains and combs
demonstrate logarithmic depth when tree height is linear; stars demonstrate it
when degree is linear.

\begin{table}[H]
\centering
\caption{Primitive operation levels for the tested topologies.}
\label{tab:depths}
\small
\setlength{\tabcolsep}{4.2pt}
\begin{tabular}{lrrrrr}
\toprule
Topology & $16$ & $64$ & $256$ & $1024$ & $4096$\\
\midrule
Chain    &  5 &  7 &  9 & 11 & 13\\
Balanced & 10 & 16 & 22 & 28 & 34\\
Comb     & 11 & 17 & 23 & 29 & 35\\
Star     &  6 &  8 & 10 & 12 & 14\\
\bottomrule
\end{tabular}
\end{table}

\section{Scenario-tree MPC example}
\label{sec:case-studies}

\begin{figure*}[!t]
\centering
\includegraphics[width=\textwidth]{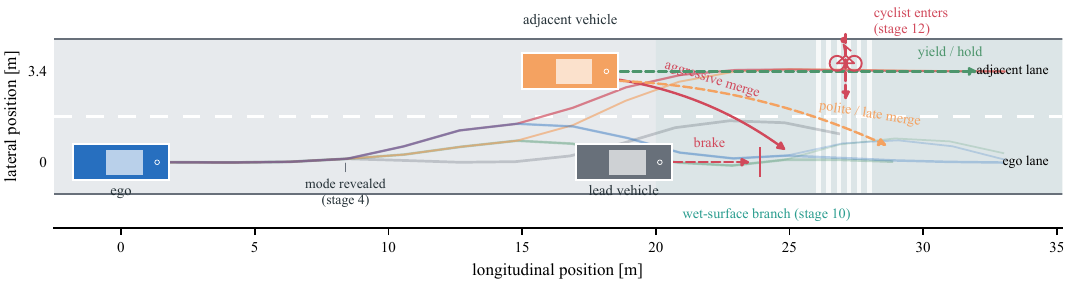}
\caption{Bird's-eye schematic and optimized ego-policy paths.  At stage~4, the
adjacent vehicle yields or merges, or the lead vehicle brakes.  Later branches
resolve the adjacent-vehicle reaction, road surface, and cyclist entry.
Vehicle glyphs and annotations are schematic; the road geometry and ego paths
are generated from the case-study data.}
\label{fig:driving-world}
\end{figure*}

We use a constructed autonomous-driving example to exercise irregular
topology, nonlinear dynamics, and inequality constraints, and to demonstrate
integration of the parallel kernel in a complete nonlinear solver.

The nonlinear program has the edge-local form
\begin{align}
\min_{x,u}\quad&
 \sum_{i\in\V}\pi_i\ell_i(x_i)
 +\sum_{e=(i,j)\in\E}\pi_j\ell_e(x_i,u_e),\label{eq:case-ocp}\\
\text{s.t.}\quad&
x_j=f_e(x_i,u_e),\notag\\[-0.2em]
&g_e(x_i,u_e)=0,\qquad h_e(x_i,u_e)\leq0.\notag
\end{align}
Node-local constraints are included where needed.  Here $\pi_i$ is the
probability of reaching node $i$, and conditional probabilities conserve mass
at every branch.  We solve the problem in \cref{eq:case-ocp} using the open-source
\texttt{primal-dual-lipa} implementation~\cite{primalduallipa}, with
\texttt{use\_parallel\_lqr=True}.  All runs use double precision and one
iterative-refinement step.

The state is longitudinal position, lateral position, speed, and heading,
$x=(s,\ell,v,\psi)$; acceleration and yaw rate form the two-dimensional control.
A nonlinear kinematic model contains $v\cos\psi$ and $v\sin\psi$.  The base
instance has a $4\,$s horizon with $0.25\,$s time steps, $N=102$ nodes, 13
leaves, height 16, and maximum out-degree three.  It constrains lateral
position to $[-1.2,4.6]\,$m, speed to $[2,13]\,$m/s, acceleration to
$[-3,2.2]\,$m/s$^2$, and yaw rate to $[-0.65,0.65]\,$rad/s.  Agent intent
is represented by mutually exclusive traffic hypotheses: at stage~4 the
adjacent vehicle yields or merges, or a lead vehicle brakes.  Selected branches
later reveal aggressive, polite, or late merge response, road friction, and
cyclist entry, as illustrated in \cref{fig:driving-world}.
The leaves terminate between stages 13 and 16.  Expected tracking, effort,
and smooth collision-avoidance penalties are minimized subject to road, speed,
acceleration, and yaw-rate bounds.  This stylized construction reflects the
delayed multimodal information pattern in branch MPC
~\cite{chen2022branch,bouzidi2025contingency}.

The primal--dual interior-point method converged in 53 iterations.  We
independently recomputed stationarity, feasibility, inequality-violation, and
perturbed-complementarity residuals from the returned primal and dual
variables.  The largest absolute component of these residuals was
$\epsilon_\infty=9.44\times10^{-8}$ at a final barrier parameter of
$7.50\times10^{-11}$.  As a structural cross-check, the generated plan contains
101 rakes and compressions, one for every nonroot node.

The policy geometry in \cref{fig:driving-world} visualizes
nonanticipativity: root-to-leaf curves coincide before an observation and
separate only after their information histories differ.

\begin{figure}[H]
\centering
\includegraphics[width=\columnwidth]{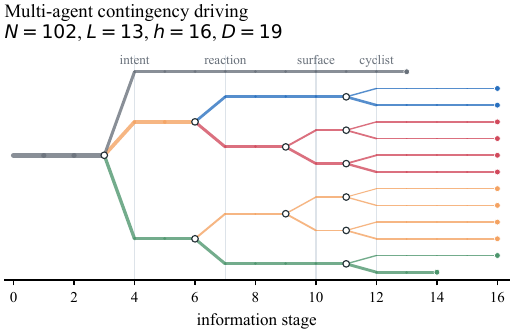}
\caption{Irregular information tree in the autonomous-driving example.  Leaves
are ordered depth first; line width is proportional to the square root of
scenario probability, white circles mark branching nodes, and colored dots
mark leaves.  $L$, $h$, and $D$ denote the numbers of leaves, tree height, and
primitive dependency levels.}
\label{fig:driving-topology}
\end{figure}

\section{Benchmarks}
\label{sec:benchmarks}

\subsection{Synthetic tree-LQR scaling}
\label{sec:synthetic-benchmarks}

\Cref{tab:timings} compares the sequential dual-regularized tree Riccati
recursion on the CPU (``CPU seq.'') with the parallel algorithm on the GPU
(``GPU par.'').  The CPU baseline follows the tree levels sequentially and
therefore has span proportional to the tree height, whereas the parallel
algorithm has logarithmic span.  We also execute the parallel algorithm on the
CPU; its times, which we omit from the table, enter the GPU/CPU ratios below to
isolate the effect of changing backends.

The measurements used a Kaggle-hosted virtualized 2.00-GHz Intel Xeon CPU
(two cores and four hardware threads) and a 16-GiB NVIDIA Tesla P100 PCIe GPU.
Each generated instance has state dimension $n=8$, control dimension $m=2$,
$Q_i=2I$, $R_e=I$, and $\Delta_i=0.02I$; the remaining data are fixed
pseudorandom arrays.  CPU and GPU runs use identical instances.  The
measurements exclude problem generation, topology planning, backend
initialization, data transfer, and JIT compilation.  Each timed call is
synchronized.  We separate factorization from an affine right-hand-side solve
because MPC algorithms commonly reuse one factorization for several right-hand
sides.  CPU measurements for the comb are omitted because XLA compilation of
the unrolled irregular plan was prohibitively slow; this is a compilation
limitation of the present JAX realization rather than an algorithmic
restriction.  These results therefore characterize repeated execution with a
reused plan and compiled executable, not one-shot latency; the excluded setup
and compilation costs can dominate a first call.

At $N=8192$, GPU factorization is $11.08\times$, $12.26\times$, and
$15.14\times$ faster than sequential CPU factorization for the chain,
balanced tree, and star, respectively.  The corresponding reused-factor solve
speedups are $19.06\times$, $20.16\times$, and $24.33\times$.  Moving the
same parallel workload from CPU to GPU gives factorization speedups of
$16.80\times$, $14.47\times$, and $13.95\times$, respectively.  The widening
gap with $N$ is consistent with the logarithmic dependency depth exposing
increasing accelerator parallelism.

\begin{table}[H]
\centering
\caption{Post-compilation FP64 times in milliseconds.  Entries are medians of
five synchronized executions; solve reuses the factorization.}
\label{tab:timings}
\scriptsize
\setlength{\tabcolsep}{2.8pt}
\begin{tabular}{llrrrr}
\toprule
& & \multicolumn{2}{c}{Factorization} & \multicolumn{2}{c}{Solve}\\
\cmidrule(lr){3-4}\cmidrule(lr){5-6}
Topology & $N$ & CPU seq. & GPU par. & CPU seq. & GPU par.\\
\midrule
Chain
 & 128  & 4.500 & 4.005 & 1.459 & 0.835\\
 & 512  & 17.943 & 6.073 & 5.098 & 1.205\\
 & 2048 & 70.565 & 10.951 & 20.036 & 2.035\\
 & 4096 & 145.712 & 16.475 & 37.421 & 3.086\\
 & 8192 & 280.218 & 25.288 & 87.143 & 4.572\\
Balanced
 & 128  & 2.968 & 1.727 & 0.926 & 0.656\\
 & 512  & 11.760 & 2.551 & 2.980 & 0.854\\
 & 2048 & 45.492 & 4.821 & 7.576 & 1.272\\
 & 4096 & 88.260 & 7.872 & 16.019 & 1.732\\
 & 8192 & 167.452 & 13.660 & 48.052 & 2.384\\
Comb
 & 128  & --- & 2.509 & --- & 0.830\\
 & 512  & --- & 4.123 & --- & 1.301\\
 & 2048 & --- & 7.219 & --- & 1.905\\
 & 4096 & --- & 10.787 & --- & 2.661\\
 & 8192 & --- & 17.520 & --- & 3.848\\
Star
 & 128  & 2.994 & 0.909 & 0.865 & 0.599\\
 & 512  & 11.033 & 1.464 & 2.497 & 0.630\\
 & 2048 & 48.600 & 3.397 & 7.944 & 0.919\\
 & 4096 & 86.903 & 6.339 & 15.516 & 1.324\\
 & 8192 & 177.473 & 11.724 & 48.072 & 1.976\\
\bottomrule
\end{tabular}
\end{table}

\subsection{Autonomous-driving refinement}
\label{sec:driving-benchmarks}

We refine the temporal discretization of the autonomous-driving problem while
holding its continuous-time definition fixed: the horizon remains $4\,$s,
branching and early-termination times are rounded to the nearest grid point,
and running costs are scaled with the step length.  Because the tree size
changes in discrete increments, the targets $128$, $512$, $2048$, and $8192$
give the closest available sizes in \cref{tab:driving-solve-timings}.  Every
instance has 13 leaves and dimensions $(n,m)=(4,2)$.

\begin{table}[H]
\centering
\caption{Post-compilation autonomous-driving solve times in milliseconds
(medians of five synchronized executions).  ``Iter.'' and ``LS'' denote outer
and aggregate logical line-search iterations; $h$ is tree height.}
\label{tab:driving-solve-timings}
\scriptsize
\setlength{\tabcolsep}{2.7pt}
\begin{tabular}{rrrrrrr}
\toprule
$N$ & $h$ & Iter. & LS & CPU par. & GPU par. & Speedup\\
\midrule
129  & 21   & 56  & 93   & 410.77   & 264.56  & $1.55\times$\\
511  & 90   & 36  & 55   & 963.80   & 273.39  & $3.53\times$\\
2044 & 366  & 181 & 1101 & 16277.34 & 1914.49 & $8.50\times$\\
8185 & 1470 & 35  & 35   & 12242.73 & 830.77  & $14.74\times$\\
\bottomrule
\end{tabular}
\end{table}

\begin{table}[H]
\centering
\caption{Autonomous-driving refinement: post-compilation Newton-system times
in milliseconds.  The solve reuses the factorization.}
\label{tab:driving-kkt-timings}
\scriptsize
\setlength{\tabcolsep}{3.2pt}
\begin{tabular}{rrrrrrr}
\toprule
& \multicolumn{3}{c}{Factorization} & \multicolumn{3}{c}{Solve}\\
\cmidrule(lr){2-4}\cmidrule(lr){5-7}
$N$ & CPU par. & GPU par. & Speedup & CPU par. & GPU par. & Speedup\\
\midrule
129  & 3.333 & 3.550 & $0.94\times$ & 1.015 & 1.506 & $0.67\times$\\
511  & 9.772 & 4.648 & $2.10\times$ & 2.521 & 1.821 & $1.38\times$\\
2044 & 37.784 & 6.712 & $5.63\times$ & 8.504 & 2.338 & $3.64\times$\\
8185 & 142.116 & 13.252 & $10.72\times$ & 29.632 & 3.792 & $7.81\times$\\
\bottomrule
\end{tabular}
\end{table}

CPU and GPU runs followed identical nonlinear paths: their outer and
line-search counts agree in every row, their objectives agree to the displayed
precision, and their independently recomputed residuals are at most
$1.94\times10^{-8}$.  The nonmonotone complete-solve time is therefore a
property of the discretized nonlinear problems, not a backend discrepancy.  In
particular, the $N=2044$ instance requires 181 outer and 1101 logical
line-search iterations, whereas the $N=8185$ instance accepts its first trial
in each of 35 outer iterations.  Thus \cref{tab:driving-kkt-timings} provides
the controlled scaling comparison for the LQR kernel, whereas
\cref{tab:driving-solve-timings} measures end-to-end integration and also
reflects discretization-dependent nonlinear convergence.

At $N=8185$, moving the same rake--compress workload from CPU to GPU accelerates
the Newton factorization by $10.72\times$ and the reused-factor solve by
$7.81\times$.  Moving the complete primal--dual interior-point algorithm to the GPU reduces
the end-to-end nonlinear-solve time by $14.74\times$.  At $N=129$,
GPU overhead instead makes the isolated factorization and solve slightly slower.
The crossover and widening speedups agree with the synthetic results: accelerator
parallelism becomes useful only after the tree is large enough to amortize
launch and scheduling overhead.

\section{Conclusion}

We developed an algebraically exact rake--compress Riccati solver for arbitrary
scenario trees.  Its conditional-value algebra is closed under leaf and
unary-node elimination, permits positive-semidefinite dual regularization, and
recovers the complete primal--dual solution.  A deterministic static scheduler
guarantees logarithmic span independently of tree height and maximum
out-degree; given its reusable topology plan, factorization and each solve
require linear numerical work and storage.  An exact linear-size lifting covers
the usual scenario-tree node-control convention.  The released JAX packages
provide the contraction--expansion algorithm, the dual-regularized LQR solver,
and a primal--dual interior-point method for tree-structured optimal-control
problems.  Benchmarks show increasing GPU speedups with tree size, and an
autonomous-driving example demonstrates integration within nonlinear
scenario-tree MPC.

\section*{Acknowledgment}

OpenAI Codex, using the GPT-5.6 Sol model, assisted with drafting and editing
this manuscript and implementing some of its methods.  The author reviewed and
verified all resulting artifacts.

\bibliographystyle{IEEEtran}
{\small\bibliography{references}}

\begin{IEEEbiography}[{\includegraphics[width=1in,height=1.25in,clip,keepaspectratio]{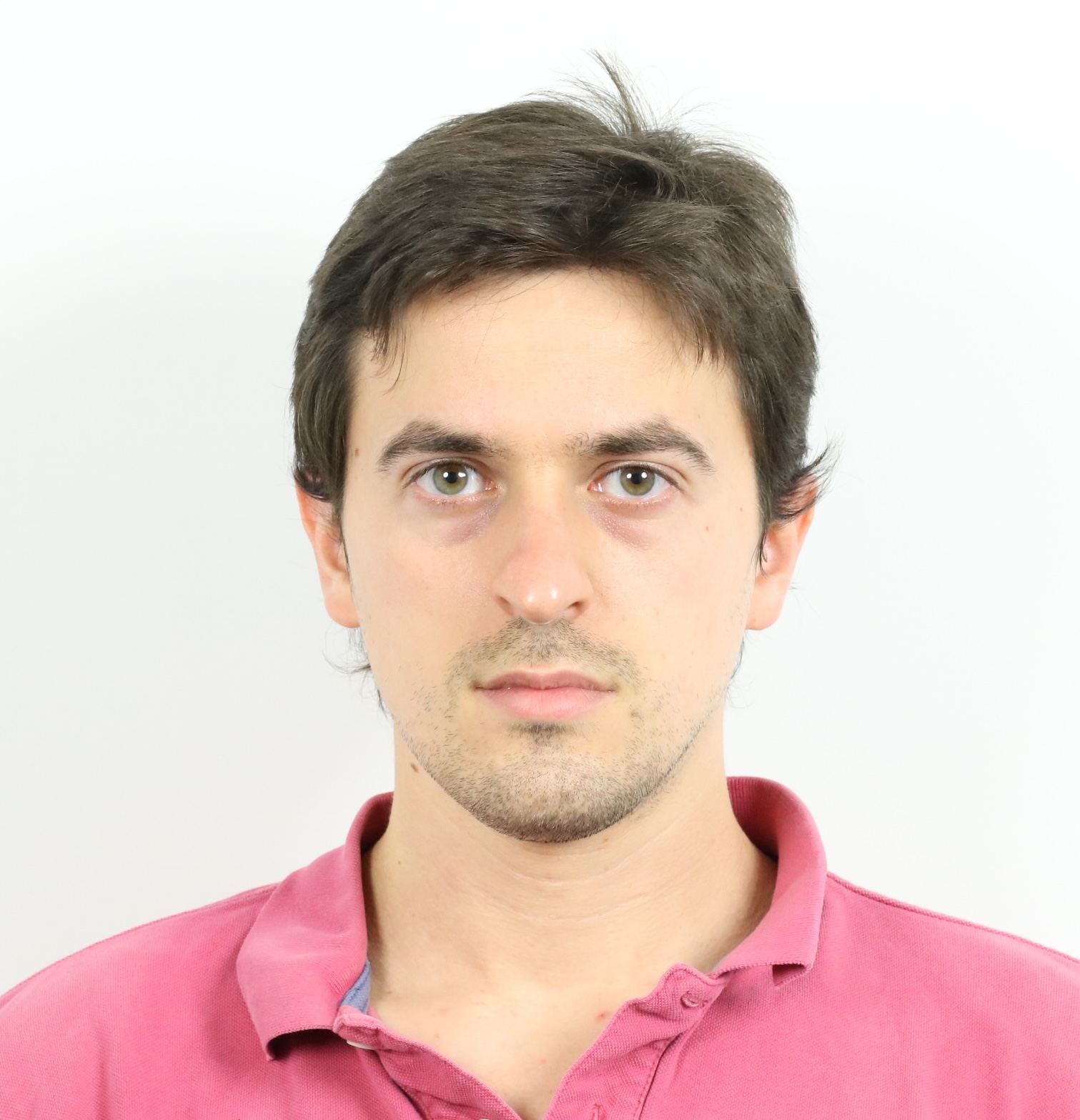}}]{João Sousa-Pinto}
João Sousa-Pinto has been working on motion planning, trajectory optimization,
and machine learning since 2017.  Over this period, he held positions at Tesla
AI, Apple SPG, drive.ai, and Baidu USA.  He was at the University of Oxford
between 2012 and 2017, having graduated from the MSc Computer Science (2013)
and DPhil Computer Science (2017) programmes.  Prior to that, João graduated
from the BSc Mathematics programme at the University of Porto, where he was a
recipient of the Calouste Gulbenkian New Talents in Mathematics scholarship.
\end{IEEEbiography}

\end{document}